\documentclass[11pt,a4paper]{article}
\usepackage[T1]{fontenc}
\usepackage{amsmath}
\usepackage{amssymb}
\usepackage{latexsym}
\usepackage{mathrsfs}
\usepackage{graphicx}
\usepackage{amsthm}
\usepackage{color}

\def\Ran{\operatorname{\rm Ran}}
\def\Re{\operatorname{\rm Re}}
\def\Im{\operatorname{\rm Im}}
\def\span{\operatorname{\rm span}}

\def\Dom{\operatorname{\rm Dom}}

\def\supp{\operatorname{\rm supp}}

\def\subsetarrow{\subset\kern-9.0pt\raise-2.95pt\hbox{$\to$}}
\newtheorem{Thm}{Theorem}
\newtheorem{Cor}{Corollary}
\newtheorem{Lem}{Lemma}
\newtheorem{Prop}{Proposition}

\newtheorem{Rem}{Remark}
\newtheorem{Ex}{Example}
\title {Persistence of embedded eigenvalues}
\author{Shmuel Agmon, Ira Herbst and Sara Maad Sasane}
\begin{document}
\maketitle \noindent
\begin{abstract}
   \noindent We consider conditions under which an embedded eigenvalue of a self-adjoint
   operator remains embedded under small perturbations.  In the case of a
   simple eigenvalue embedded in continuous spectrum of multiplicity $m < \infty$
   we show that in favorable situations, the set of small perturbations  of a
   suitable Banach space which do not remove the eigenvalue form a smooth
   submanifold of co-dimension $m$. We also have results regarding the cases when the
   eigenvalue  is degenerate or when the multiplicity of the continuous spectrum is infinite.
\end{abstract}

\section{Introduction}\noindent
An eigenvalue in the continuous spectrum of an operator typically disappears under small perturbations. Or if there is enough analyticity for some sort of analytic continuation of the resolvent, it typically becomes a resonance, that is a pole in the analytic continuation of certain matrix elements of the resolvent.

The simplest mechanism which has been used to prove that embedded eigenvalues disappear under perturbation is Fermi's Golden Rule, which in physics gives the lifetime of the decaying unperturbed state. See for example \cite{PS92a, PS85, PS92b, AHS89}. See also the earlier work
\cite{yC82,yC83}, which relies on analytic continuation. For the basic mathematical ideas behind Fermi's Golden Rule, see \cite{bS73}.

On the contrary, in this article we are interested in the structure of the set of perturbations which do not remove an embedded eigenvalue. No analyticity assumptions are made to allow the analytic continuation of the resolvent. A simple example of the kind of theorem we are after is given in \cite{CHM02}: Suppose $V$ is real and in $L^1(\mathbb R)$. Consider the self-adjoint operator in $L^2(\mathbb R)$ given
by
\begin{equation*}
   H = -\frac{d^2}{dx^2}+ \lambda \frac{\sin(kx)}{x}+ V(x),
\end{equation*}
where $\lambda> k>0$. Then the set of such $V$ for which $H$ has an embedded eigenvalue is a smooth codimension $1$ submanifold of real $L^1(\mathbb R)$. Such a global result is difficult to obtain in more general cases. We restrict ourselves to small bounded perturbations of a given operator, but the methods can be extended to small $H$-bounded perturbations.

In Sections \ref{S:main} to \ref{S:finmult}, we consider a simple eigenvalue embedded in continuous spectrum of multiplicity $m<\infty$. Under favorable assumptions including the smoothness of the boundary values of the resolvent (after the pole term corresponding to the embedded eigenvalue has been
removed) we show that small perturbations which do not remove the eigenvalue form a smooth submanifold (of appropriate Banach spaces) of codimension $m$. See Theorem \ref{T:main}. In Section \ref{S:applications}, we give two applications of this theorem.

In Section \ref{S:degenerate}, a smooth manifold of perturbations of codimension $m+n-1$ is shown not to remove a degenerate eigenvalue of multiplicity $n$ embedded in continuous spectrum of multiplicity $m$. The set of small perturbations which do not remove the degenerate eigenvalue is a much larger set, but its structure is not known.

In Section \ref{S:infinmult} we give a weak theorem, but one which covers a very general class of operators of the form $-\Delta + V$, where $V$ is a real function on $\mathbb R^n$. This theorem shows that the set of small local perturbations which do not remove a {(simple or degenerate)} eigenvalue is quite large. Of course if $n\ge 2$, the continuous spectrum will in general have infinite multiplicity.

See \cite{DMS10} for another approach to the problem where the continuous spectrum of the operators involved has infinite multiplicity. In \cite{DMS10}, the structure of the set of local perturbations which do not remove an embedded simple eigenvalue is determined for a specific example.

\section{Assumptions and result in the case of a simple eigenvalue and finite multiplicity of the continuous spectrum}\label{S:main}\noindent
Let $\mathcal H$ be a Hilbert space, and let $C:\mathcal H\to \mathcal H$ be an antiunitary involution, i.e. a conjugate-linear
mapping satisfying $C^2 = I$ and $\langle Cf,Cg\rangle=\overline{\langle f,g\rangle}$.
An element $f\in \mathcal H$ is said to be {\em real} if $C f = f$, and
we say that an operator $H$  on $\mathcal H$ is {\em real} if $H C = C H$.
We assume that:
\begin{enumerate}
   \item [(H1)] $H$ is a real, self-adjoint operator acting in $\mathcal H$.
\end{enumerate}
We introduce a scale of Hilbert spaces $\mathcal H_s$ for $s\in \mathbb R$ such that $\mathcal H_0 = \mathcal H$, the dual space of $\mathcal H_s$ is
$\mathcal H_{-s}$ (using the inner product of $\mathcal H$) and $\mathcal H_s$ is continuously embedded in $\mathcal H_t$ for $s\ge t$.
We also assume that $\mathcal H_s$
is dense in $\mathcal H_t$ if $s>t$. If $s\ge 0$, then $\mathcal H_s\subset \mathcal H\subset \mathcal H_{-s}$. We denote the inner product of
$\mathcal H$ and also the duality pairing of $\mathcal H_s$ with $\mathcal H_{-s}$ by $\langle \cdot,\cdot\rangle$. For notational simplicity we assume that for $s>0,\, \|f\|_{\mathcal H} \le \|f\|_{\mathcal H_s}$.  We assume that $C:\mathcal H_s\to \mathcal H_s$ is
bounded for
every $s\in \mathbb R$. It then follows that

\begin{equation*}
   \langle C f,C g\rangle = \overline{\langle f,g\rangle}
\end{equation*}
for $f\in \mathcal H_{-s}$ and $g\in \mathcal H_s$.

Let $\sigma_{pp}(H)$ be the pure point spectrum of $H$, i.e. the set of eigenvalues of $H$.
\begin{enumerate}
   \item [(H2)] $H$ has an eigenvalue at $\lambda = \lambda_0$ of finite multiplicity which is embedded in the continuous spectrum and isolated in $\sigma_{pp}(H)$,
   and the corresponding eigenspace is a subspace of $\cap_{s\ge 0}\mathcal H_s$.
\end{enumerate}
The condition that the eigenspace is a subspace of $\cap_{s\ge0}\mathcal H_s$ can be relaxed, and it is enough that
it is a subspace of $\mathcal H_{s_*}$ for
 a certain $s_*>0$. In examples, the condition can be checked by using methods from \cite{FMS10,FH82, CGH10, MW10}.

We denote by $P_0$ the orthogonal projection in $\mathcal H$ onto the eigenspace of $H$ corresponding to the eigenvalue $\lambda_0$. Let $\overline H
:= H + P_0$. If $H$ satisfies conditions (H1) and (H2) then the continuous spectra of $H$ and $\overline
H$ coincide, and $\overline H$ does not have any eigenvalues in a neighborhood of $\lambda_0$ (see Proposition \ref{P:H bar H}).

\begin{enumerate}
   \item [$\text{(H3)}_k$] There exist $k\ge 0$ and $s_1 \ge 0$ and  such that for any $s>s_1$ there is a $\delta_1>0$
   such that the norm limits  $\lim _{\epsilon \downarrow 0} (\overline H - \lambda \pm i\epsilon)^{-1} = (\overline H - \lambda \pm i0)^{-1}$ \text {exist in} \,$\mathcal  L(\mathcal H_s,\mathcal H_{-s})$ and are $C^{k}$
   in $\lambda$ in the norm topology for
   $\lambda \in (\lambda_0-\delta_1,\lambda_0+\delta_1)$.
\end{enumerate}
In examples, $\text{(H3)}_k$ can be verified using methods from \cite{eM81, JMP84}, as is done in Example \ref{Ex:line} and \ref{Ex:cylinder} of this paper. $\text{(H3)}_0$ is called the limiting absorption principle for $\overline H$.
Suppose that $\text{(H3)}_k$ holds.  We will consider perturbations $W$ in the space $X_s$, where $X_s$ is a real Banach space
whose elements are bounded self-adjoint operators on $\mathcal H$ and
such that
\begin{equation*}
   X_s\subset\{W\in \mathcal L(\mathcal H_{-s},\mathcal H_{s});\; W \text{ is real and self-adjoint on }\mathcal H \},
\end{equation*}
 where $s>s_1$ and the inclusion is continuous.

If $H$ satisfies (H1), (H2), and $\text{(H3)}_0$, we introduce the notation
\begin{equation*}
   \delta(\overline H - \lambda) := \frac{1}{2\pi i} \left((\overline H - \lambda - i0)^{-1} - (\overline H - \lambda +
   i0)^{-1}\right).
\end{equation*}
Note that if $\lambda\ne \lambda_0$ but $|\lambda-\lambda_0|$ is small, then $\delta(\overline H - \lambda) = \delta(H - \lambda)$.
The multiplicity of the continuous spectrum of $H$ at $\lambda$  is by definition the dimension of $\Ran \delta(\overline H - \lambda)\subset \mathcal H_{-s}$.  Using the density of $\mathcal H_s $ in $ \mathcal H_t$ for $s> t$ it is easy to show that the multiplicity is independent of $s$ for $s > s_1$.
\begin{enumerate}
   \item [(H4)] The multiplicity of the continuous spectrum  of $H$ in $(\lambda_0-\delta_1,\lambda_0+\delta_1)$ is $m<\infty$.
\end{enumerate}
Our last assumption is a condition that the set of perturbations is not too small. We will eventually need one version of this condition (H5) when $\lambda_0$ is a simple eigenvalue of $H$, and a stronger condition (H5') when $\lambda_0$ is a degenerate eigenvalue.

\begin{enumerate}
   \item [(H5)] $\lambda_0$ is a simple eigenvalue and $\varphi_0$ is a corresponding real normalized eigenvector of $H$. The
   complex linear span of $\{\delta(\overline H-\lambda_0)W\varphi_0;\; W\in X_s\}$ is $\Ran \delta(\overline H-\lambda_0)$.
   \item [(H5')] There exists a real vector
   $\psi_1\in \Ran P_0$ such that the complex linear span of $\{W\psi_1;\; W\in X_s\}$ is dense in $\mathcal H_s$.
\end{enumerate}

In Section \ref{S:applications} we give examples of operators for which the assumptions (H1)-(H5) are satisfied.
For $\delta>0$, let
$\mathcal M_{\delta,s}$ be the set
\begin{equation*}
   \mathcal M_{\delta,s} := \{ W\in X_s;\; \text{there exists a }\lambda\in (\lambda_0-\delta,\lambda_0+\delta)\text{ such that
   }\lambda\text{ is an eigenvalue of }
   H + W\}.
\end{equation*}
\begin{Thm}\label{T:main}
   Suppose that $H$ satisfies conditions (H1), (H2), $\text{(H3)}_k$ with $k \ge 1$ and (H4) and (H5). Let $s_1$ and  $m$ be as in assumptions
   $\text{(H3)}_k$ and (H4), and let $s>s_1$ and $\dim X_s \ge m$.
   Then there exist a number
   $\delta>0$ and a neighborhood $\mathcal O$ of $0$ in $X_s$ such that $\mathcal M_{\delta,s}\cap \mathcal O$ is a $C^k$ manifold in $X_s$ of codimension $m$. Moreover, if $W\in\mathcal M_{\delta,s}\cap\mathcal O$, then $H+W$ has exactly $1$ eigenvalue in the interval $(\lambda_0-\delta,\lambda_0+\delta)$, and it is simple.
\end{Thm}
\section{Some preliminary lemmas}\label{S:prel}\noindent
   Some of the propositions of this section are similar  to results found elsewhere in the literature, see e.g. \cite{AHS89,jH72,jH74,tK66}.  Furthermore some of the propositions and lemmas needed for proving Theorem \ref{T:main} are valid without all of the assumptions (H1) -- (H5).

We first remark that we have not assumed a condition of uniformity in $\text{(H3)}_k$.  That this assumption is unnecessary follows from the following lemma:

\begin{Lem} \label{Lem:uniform}
Let $k\ge 0$ and suppose $\text{(H3)}_k$ holds.  Then in any compact subinterval $J_1 \subset (\lambda_0 - \delta_1, \lambda_0 + \delta_1)$ and any $j \le k$ the convergence of $\displaystyle{\frac {d^j(\overline H - \lambda \pm i\epsilon)^{-1}}{d\lambda^j}}$ to its boundary value $\displaystyle{\frac {d^j(\overline H - \lambda \pm i0)^{-1}}{d\lambda^j}}$ is uniform for $\lambda \in J_1$.
\end{Lem}

\begin{proof}
If $J$ is any compact subinterval of $(\lambda_0 - \delta_1, \lambda_0 + \delta_1)$, let $Q_J = \{z\in \mathbb{C};\; \Im z >0, \; \Re z\in J\}$. For $z$ with $\Im z>0$, define $F(z) = (\overline H-z)^{-1}$.
We shall consider $F$ as an operator valued function with values in $\mathcal L( \mathcal H_s, \mathcal H_{-s})$, $s > s_1$. Note that $F$ is analytic in the half-plane $\Im z>0$ and satisfies $\|F(z)\|_{\mathcal L(\mathcal H_s,\mathcal H_{-s})}\le C(\Im z)^{-1}$. We shall show that $F$ is bounded in $Q_J$ and that $F$ admits a continuous extension to $\overline {Q_J}$.
For this purpose, if $f \in \mathcal H_s$ let
\begin{equation*}
 h(z) = (\langle F(z)f,f\rangle +i)^{-1}
 \end{equation*}
 and note that $|h(z)| \leq 1$.  A well known theorem asserts that
$$h(x) := \lim_{y \downarrow 0} h(x+iy)$$
exists for a.e $x\in \mathbb R$ and that
\begin{equation} \label{E:Poisson1}
 h(x+iy) = \frac{1}{\pi}\int_{-\infty}^\infty P(x-t,y)h(t) dt
\end{equation}
where $P$ is the Poisson kernel
$$P(x,y) = \frac{y}{x^2 + y^2}.$$
(see for example \cite {R}, Theorems 11.24 and 11.30). Note that it follows from assumption $\text{(H3)}_k$ (and the definition of $h(z)$) that the limit defining $h(x)$ exists for all $x\in J$, that $h(x)$ is continuous for $x\in J$, and that $h(x)\ne 0$ for $x\in J$.
Hence, it follows from the representation \eqref{E:Poisson1} that $h(z)$ admits a continuous extension to $\overline{Q_J}$ and that for some $\delta>0$,
\[
   |h(z)| \geq \delta, \qquad z \in Q_J \cap \{z: \Im z \leq 1\}
\]
and thus by polarization $\langle F(z)f,g\rangle$ is continuous on $\overline{Q_J}$ for all $f,g \in \mathcal {H}_s$.  In particular the uniform boundedness principle implies that $\|F(z)\|_{\mathcal L(\mathcal H_s,\mathcal H_{-s})} \leq C_J$ for $z\in \overline{Q_J}$, where $J$ is any compact subinterval of $(\lambda_0 - \delta_1, \lambda_0 + \delta_1)$.

With $J$ and $Q_J$ as above, let $\zeta(t), t \in \mathbb{R}$ be a $C^1$ curve in $\overline{Q_J}\cup \{\Im z>0\}$ satisfying $\zeta(t) = t$ for $t\in J$ and  $\zeta(t) = t +i$ for $t \notin (\lambda_0 - \delta_1, \lambda_0 + \delta_1)$ and $\Im \zeta(t) > 0$ if $t \notin J$. Let $z=x+iy$. Integrating the function $F(\zeta)(2\pi i(\zeta-z)(\zeta-\overline z))^{-1}$ along the curve $\zeta(\cdot)$ we obtain (by the residue theorem) the representation
\begin{equation} \label{eqn:Poisson2}
F(x+iy) = \frac{1}{\pi}\int_{\mathcal {C}}\frac{y F(\zeta)}{(x-\zeta)^2 +y^2} d\zeta,
\end{equation}
where $\mathcal {C}$ is the curve $\zeta(\cdot)$.  It is clear from (\ref{eqn:Poisson2}) and the continuity of $F$ on $(\lambda_0 - \delta_1, \lambda_0 + \delta_1)$ that $F$ is continuous in the topology of $\mathcal L( \mathcal H_s, \mathcal H_{-s})$ for $x+iy$ on or above $\mathcal C$.

Differentiating (\ref{eqn:Poisson2}) we obtain

$$\frac{\partial F(x+iy)}{\partial x} = \frac{1}{\pi}\int_{\mathcal {C}}\frac {\partial((x-\zeta)^2 +y^2)^{-1}}{\partial x}y F(\zeta) d\zeta$$
If $J =[a,b]$ we integrate by parts on $[a,b]$ to obtain

\begin{equation} \label{eqn:byparts}
\begin{aligned}
\frac{\partial F(x+iy)}{\partial x} &= \frac{1}{\pi}\int_{\mathcal {C}\setminus [a,b]}\frac {\partial((x-\zeta)^2 +y^2)^{-1}}{\partial x}y F(\zeta) d\zeta + \frac{1}{\pi}\int_{[a,b]}\frac{y F'(\zeta)}{(x-\zeta)^2 +y^2}d\zeta \\
&\qquad + \frac{F(b) y}{\pi((x-b)^2 +y^2)} - \frac{F(a) y}{\pi((x-a)^2 +y^2)}.
\end{aligned}
\end{equation}
If $[a',b'] \subset (a,b)$  it follows that uniformly for $x \in [a',b']$, $\displaystyle{\lim_{y \downarrow 0} \frac{\partial F(x+iy)}{\partial x}} = F'(x)$.  Note also the equality for the one-sided derivative
$$ \left.-i\frac {\partial F(x+iy)}{\partial y}\right|_{y=0} = F'(x) $$ which we will use in Proposition \ref{P:ev}.  This follows from taking $y' \downarrow 0$ in
$$ F(x+iy) - F(x+iy') = \int_{y'}^y \frac{\partial F(x+it)}{\partial t} dt = i\int_{y'}^y \frac{\partial F(x+it)}{\partial x} dt$$

We have thus proved the uniform convergence for $j=0,1$.  In the same way, differentiating (\ref{eqn:byparts}) as many times as necessary and integrating by parts, the lemma follows for limits from the upper half plane.  A similar proof works for the lower half plane.
\end{proof}

The following proposition is basically a corollary of Lemma \ref{Lem:uniform}.
\begin{Prop} \label{Prop:uniform2}
Suppose that $\text{(H3)}_k$ holds for some $k \ge 0$ and $s > s_1$.  Then given a compact subinterval $J_1 \subset (\lambda_0 - \delta_1, \lambda_0 + \delta_1)$ there exists $\gamma = \gamma_{J_1} > 0$ so that if \, $\|W\|_{\mathcal  L(\mathcal H_{-s},\mathcal H_{s})} < \gamma$ and $j \le k$ the limits
\begin{equation}
 \lim_{\epsilon \downarrow 0}\frac{d^j(\overline H +W - \lambda \pm i\epsilon)^{-1}}{d\lambda^j} =  \frac{d^j(\overline H +W - \lambda \pm i0)^{-1}}{d\lambda^j}
\end{equation}
are uniform in $\lambda$ for $\lambda \in J_1$.
\end{Prop}
\begin{proof}
Let $z = \lambda + i \epsilon$. Define $\gamma$ so that
$$\gamma \sup\{ \|(\overline H - z)^{-1}\|_{\mathcal  L(\mathcal H_s,\mathcal H_{-s})};\; \Re z \in J_1,\, \Im z \ge 0\} < 1.$$
Then
\begin{equation}\label{E:pre-boundary}
      (\overline H +W - z)^{-1} = (\overline H - z)^{-1}(I +W (\overline H - z)^{-1})^{-1}
\end{equation}
Then noting that $A \mapsto A^{-1}$ is $C^{\infty}$ we can differentiate (\ref{E:pre-boundary}) $k$ times using the chain rule and take limits as $\epsilon \downarrow 0$ to obtain the result.
\end{proof}

We note that by Theorem XIII.20 of \cite{RS78}, if condition $\text{(H3)}_0$ holds and $s > s_1$, then for $\|W\|_{\mathcal  L(\mathcal H_{-s},\mathcal H_{s})} < \gamma_{J_1}$ where $J_1$ is a compact subinterval of $(\lambda_0 - \delta_1, \lambda_0 + \delta_1)$,
$\overline H$ and $\overline H+W$ have purely absolutely continuous spectrum in $J_1$.

   Recall that $P_0$
   is the orthogonal projection in $\mathcal H$ onto the eigenspace corresponding
   to $\lambda_0$ and that $\overline H = H + P_0$.  We have the following proposition:
   \begin{Prop}\label{P:H bar H}
   Let $H$ satisfy (H1), (H2), and $\text{(H3)}_0$ and assume $s>s_1$.
   Then
      \begin{enumerate}
         \item [(i)]{$\sigma_c(\overline H) = \sigma_c(H)$},
         \item [(ii)]{$\overline H$ has no eigenvalues in $(\lambda_0 -\delta_1, \lambda_0 + \delta_1)$.}
      \end{enumerate}
   \end{Prop}
   \begin{proof}
      To prove (i), let $\psi\in\Dom H\cap\Ran (I - P_0) = \Dom H \cap\ker P_0$. Then
      \begin{equation*}
         \overline H \psi = (H + P_0) \psi = H \psi.
      \end{equation*}
      In other words, on $\Dom H\cap\ker P_0$, the operators $H$ and $\overline H$ coincide. In particular, their continuous spectra
      are the same.

      According to the remark after Proposition \ref{Prop:uniform2}, $\overline H$ has purely absolutely continuous spectrum in $(\lambda_0 -\delta_1, \lambda_0 + \delta_1)$.  Thus (ii) follows.
   \end{proof}

   Our proof of Theorem \ref{T:main} is based on the study of the operator $Q(z,W)$ defined by
   \begin{equation}\label{E:Q}
      Q(z,W) = P_0(\overline H + W -z)^{-1} P_0
   \end{equation}
   for $\Im z\ne 0$ and $W\in X_s$ ($s\ge 0$).
   Note that $Q(z,W)$ is a finite dimensional operator on $\Ran P_0$, since $\lambda_0$ has finite multiplicity.

   The resolvent equation shows that for $\text{Im } z \ne 0$,
   \begin{equation}\label{E:firstres}
      (H+W-z)^{-1} = (\overline H + W - z)^{-1} + (H + W - z)^{-1} P_0 (\overline H + W - z)^{-1},
   \end{equation}
   which implies that
   \begin{equation*}
      (H + W - z)^{-1} P_0 = (\overline H + W - z)^{-1} P_0 + (H + W - z)^{-1} Q(z,W),
   \end{equation*}
   or equivalently
   \begin{equation}\label{E:I-Q}
      (H+W-z)^{-1}P_0(I-Q(z,W))=(\overline H + W -z)^{-1} P_0.
   \end{equation}
   This formula gives a one to one correspondence between the two operators $(H+W - z)^{-1} P_0$ and
   $(\overline H + W - z)^{-1} P_0$. In fact, we have
   \begin{Prop}\label{P:Q-inv}
      Suppose that (H1) and (H2) are satisfied, and for $\Im z\ne 0$ and $W\in X_s$ ($s\ge 0$) let
      $Q(z,W):\mathcal H\to \mathcal H$ be given by
      \eqref{E:Q}. Then
      $I - Q(z,W)$ is invertible.
   \end{Prop}
   \begin{proof}
      If $I - Q(z,W)$ is not invertible, then by the Fredholm alternative, there exists an $f\ne 0$ such that
      \begin{equation*}
         (I-Q(z,W))f = 0.
      \end{equation*}
      But then by \eqref{E:I-Q},
      \begin{equation*}
         (\overline H + W -z)^{-1} P_0 f = 0,
      \end{equation*}
      which implies that $P_0 f=0$. This means that $f=(I-P_0)f$, and so
      \begin{equation}\label{E:prod}
         0=(I-Q(z,W))(I-P_0)f = (I-P_0)f,
      \end{equation}
      and we see that $f=0$.
   \end{proof}
   Proposition \ref{P:Q-inv} together with equation \eqref{E:I-Q} show that
   \begin{equation*}
      (H + W - z)^{-1} P_0 = (\overline H + W -z)^{-1} P_0 (I-Q(z,W))^{-1}.
   \end{equation*}
   Then by \eqref{E:firstres} we have
   \begin{equation}\label{E:H - bar H}
      (H+W-z)^{-1} = (\overline H + W - z)^{-1} + (\overline H + W - z)^{-1} P_0 (I-Q(z,W))^{-1} P_0 (\overline H + W - z)^{-1}.
   \end{equation}
   Assuming $\text{(H3)}_0$ and $s>s_1$, it follows that
   \begin{equation}\label{E:decomp}
      Q(\lambda + i 0,W) = A(\lambda, W) + i \pi P_0 \delta(\overline H + W - \lambda) P_0,
   \end{equation}
   where
   \begin{equation}\label{E:A-def}
      A(\lambda,W):= \frac{1}{2} \left(Q(\lambda + i0,W) + Q(\lambda - i0,W)\right)
   \end{equation}
   and
   $\delta(\overline H + W - \lambda)$ is given by
   \begin{equation}\label{E:delta-def}
      \begin{aligned}
         \delta(\overline H + W - \lambda):&=
         \lim_{
         \begin{smallmatrix}
            z=\lambda + i \epsilon\\
            \epsilon\downarrow 0
         \end{smallmatrix}}
         \frac{1}{2\pi i} \left((\overline H + W - z)^{-1} - (\overline H + W -\bar z)^{-1}\right) \\
         &= \lim_{
         \begin{smallmatrix}
            z=\lambda + i \epsilon\\
            \epsilon\downarrow 0
         \end{smallmatrix}}
         (\overline H + W - z)^{-1}\frac{\epsilon}{\pi}(\overline H + W - \bar z)^{-1}.
      \end{aligned}
   \end{equation}

   \begin{Prop} \label{P:ev}
      Let $H$ satisfy (H1), (H2), and $\text{(H3)}_1$. Let $s>s_1$ be fixed, and let $J_1$ and $\gamma$ be as in Proposition \ref{Prop:uniform2} with $\|W\|_{\mathcal  L(\mathcal H_{-s},\mathcal H_{s})} < \gamma$.
      If $\lambda \in J_1$, then $1$ is an eigenvalue of $Q(\lambda + i 0, W)$ if and only if
      $\lambda$ is an eigenvalue of $H+W$.
   \end{Prop}
   \begin{proof}
      The limit $Q(\lambda + i 0,W)$ is a compact operator, so if
      $Q(\lambda + i 0,W)$ does
      not have eigenvalue $1$, then $(I-Q(\lambda + i 0, W))^{-1}$ exists, and we
      can take the limit in \eqref{E:H - bar H}. Hence also $(H+W-\lambda - i 0)^{-1}$ exists. It was shown
      above that this implies that $\lambda$ is not an eigenvalue of $H+W$.
To prove the converse, we first use Proposition \ref{Prop:uniform2} for $j=1$ to see that $Q(z,W)$ has a $C^1$ extension to the real axis for $\lambda \in J_1$.  In particular we have the Taylor expansion
      \begin{equation}\label{E:Taylor}
         Q(z,W) = Q(\lambda+i0,W) + Q_z'(\lambda+i0,W)(z-\lambda) + o(|z-\lambda|).
      \end{equation}

      Suppose that $Q(\lambda + i 0,W)f=f$ for some $f$. Then by the definition of $Q(z,W)$,
      $(I-P_0)f=0$. If $\lambda$ is not an eigenvalue of
      $H+W$, then by \eqref{E:I-Q} and \eqref{E:Taylor}
      \begin{equation*}
         \begin{aligned}
            P_0 f &= \lim_{
            \begin{smallmatrix}
               z= \lambda + i \epsilon \\
               \epsilon\downarrow 0
            \end{smallmatrix}}
            (I + P_0 (H + W - z)^{-1}) P_0 (I-Q(z,W))f \\
            &= -\lim_{
            \begin{smallmatrix}
               z= \lambda + i \epsilon \\
               \epsilon\downarrow 0
            \end{smallmatrix}}
            (I + P_0 (H + W - z)^{-1}) P_0 \left(Q'_z(\lambda + i0,W)(z-\lambda) + o(|z-\lambda|)\right)f \\
            &=-\lim_{
            \begin{smallmatrix}
               z= \lambda + i \epsilon \\
               \epsilon\downarrow 0
            \end{smallmatrix}}
            P_0 (H + W - z)^{-1}(z-\lambda) P_0 Q'_z(\lambda + i0,W) f \\
            &= P_0 E_{\{\lambda\}}(H+W) P_0 Q'_z(\lambda + i0,W)f =0
         \end{aligned}
      \end{equation*}
      since $E_{\{\lambda\}}(H+W) = 0$.  This shows that $f=0$.
   \end{proof}
   \begin{Cor}
      Let $H$ satisfy (H1), (H2), and $\text{(H3)}_1$. Let $s>s_1$ be fixed, and let $J_1$ and $\gamma$ be as in Proposition \ref{Prop:uniform2} with $\|W\|_{\mathcal  L(\mathcal H_{-s},\mathcal H_{s})} < \gamma$. It follows that if $\lambda$ is not an eigenvalue of $H+W$, then the limiting absorption principle holds for $H+W$ at $\lambda$.
   \end{Cor}
   \begin{proof}
      The statement follows from Proposition \ref{P:ev} and \eqref{E:H - bar H}.
   \end{proof}
   In the case when the multiplicity of the eigenvalue $\lambda_0$ is $1$, we get a simple expression for the
   eigenvector of $H+W$ corresponding to $\lambda$.
   \begin{Prop}\label{P:eigenvector}
      Suppose that all conditions of Proposition \ref{P:ev} are satisfied and that $\lambda_0$ is a simple eigenvalue of $H$. Suppose that $\lambda$ is an eigenvalue of $H+W$ with $\lambda \in J_1$. Then $\lambda$ is a simple eigenvalue of $H+W$ and
      a corresponding eigenvector is given by
       \begin{equation}\label{E:eigenvector-expr}
         \psi = (\overline H+W-\lambda-i0)^{-1}\varphi_0.
      \end{equation}
   \end{Prop}
   \begin{proof}
      We first prove that $\psi$ as defined by \eqref{E:eigenvector-expr} is nonzero.
      This follows from Proposition \ref{P:ev} since
      $$
        P_0 \psi = Q(\lambda+i0,W)\varphi_0=\varphi_0\ne 0.
      $$

      Next we use \eqref{E:Taylor}, and let the number $d=d(\lambda,W)$ be defined by $Q'_z(\lambda+i0,W) = d P_0$. Hence, by
      \eqref{E:H - bar H} and Proposition \ref{P:ev}, we have for $f\in \mathcal H_s$
      \begin{equation}\label{E:spectralproj}
         \begin{aligned}
            -i\epsilon (H+W&-\lambda-i\epsilon)^{-1} f = -i\epsilon (\overline H + W - \lambda -
            i\epsilon)^{-1}f \\
            &+ (d+o(1))^{-1} (\overline H + W -\lambda-i\epsilon)^{-1}
            P_0 (\overline H + W - \lambda-i\epsilon)^{-1}f
         \end{aligned}
      \end{equation}
      as $\epsilon\downarrow 0$. Choosing $f=\varphi_0$ and passing to the limit in $\mathcal H_{-s}$ as $\epsilon\downarrow 0$, it follows that $d\ne 0$, since otherwise the right hand side of \eqref{E:spectralproj}
      would blow up (using that $\psi\ne 0$ and $P_0\psi=\varphi_0$), while the left hand side tends to $E_{\{\lambda\}}(H+W)\varphi_0$. Indeed, we
      have
      \begin{equation}\label{E:eigenproj}
         E_{\{\lambda\}}(H+W)\varphi_0 = \frac{1}{d}(\overline H + W - \lambda -i0)^{-1} \varphi_0,
      \end{equation}
      which also shows that the right hand side of \eqref{E:eigenproj} belongs to
      $\mathcal H$.
      Multiplying by $d$ yields the expression
      \eqref{E:eigenvector-expr}.

      Finally, by \eqref{E:spectralproj} for a general $f\in \mathcal H_s$, and passing to the limit as $\epsilon\downarrow 0$, the right hand side is always a multiple of $\psi$, and so $\lambda$ is a simple eigenvalue of $H+W$.
   \end{proof}
   \begin{Prop}\label{P:differentiable}
      Suppose that (H1), (H2), and $\text{(H3)}_k$ are satisfied, where $k\ge 0$, and let $s>s_1$, where $s_1$ is defined in $\text{(H3)}_k$. Then for some neighborhood $J$ of $\lambda_0$ and some neighborhood $\mathcal O$ of $0 \in X_s$,
      $(\overline H + W - \lambda - i0)^{-1}: J \times \mathcal O \to \mathcal L(\mathcal H_s,\mathcal H_{-s})$ is $C^k$ as a function of $\lambda$ and
      $W$.
   \end{Prop}
   \begin{proof}
      We first compute the partial Fr\'echet derivative of $R(\lambda,W):=(\overline H + W - \lambda - i0)^{-1}$
      with respect to $W$. We will see that it is given by
      \begin{equation}\label{E:Q-der-W}
          R'_W(\lambda,W)\widetilde W= -R(\lambda,W) \widetilde W R(\lambda,W).
      \end{equation}
      Indeed, by the resolvent equation we have
      \begin{equation}\label{E:differentiable}
         \begin{aligned}
            \|R(\lambda,W +\widetilde W) &- R(\lambda,W) + R(\lambda,W)
            \widetilde W R(\lambda,W)\|_{\mathcal L(\mathcal H_s,\mathcal H_{-s})}   \\
            =&\| - R(\lambda,W+\widetilde W)\widetilde W R(\lambda,W)
            + R(\lambda,W)
            \widetilde W R(\lambda,W)\|_{\mathcal L(\mathcal H_s,\mathcal H_{-s})} \\
            =& \| R(\lambda,W+\widetilde W)\widetilde W R(\lambda,W) \widetilde W R(\lambda,W)\|_{\mathcal L(\mathcal H_s,\mathcal H_{-s})} \\
            \le& \|R(\lambda,W+\widetilde W)\|_{\mathcal L(\mathcal H_s,\mathcal H_{-s})}\|R(\lambda,W)\|_{\mathcal L(\mathcal H_s,\mathcal H_{-s})}^2
            \|\widetilde W\|_{\mathcal L(\mathcal H_{-s},\mathcal H_s)}^2.
         \end{aligned}
      \end{equation}
      Since
      \begin{equation*}
         R(\lambda,W) = R(\lambda,W+\widetilde W)(I + \widetilde W R(\lambda,W)),
      \end{equation*}
      and since  $R(\lambda,W)$ is bounded from $\mathcal H_s$ to $\mathcal H_{-s}$, we have for
      $\|\widetilde W\|_{\mathcal L(\mathcal H_{-s},\mathcal H_s)}$ small
      \begin{equation*}
         R(\lambda,W+\widetilde W) = R(\lambda,W)
         ( I + \widetilde W R(\lambda,W))^{-1},
      \end{equation*}
      so that
      \begin{equation*}
         \|R(\lambda,W+\widetilde W)\|_{\mathcal L(\mathcal H_s,\mathcal H_{-s})} \le
         \|R(\lambda,W)\|_{\mathcal L(\mathcal H_s,\mathcal H_{-s})}
         \bigl( 1 - \|\widetilde W\|_{\mathcal L(\mathcal H_{-s},\mathcal H_{s})}
         \|R(\lambda,W)\|_{\mathcal L(\mathcal H_s,\mathcal H_{-s})}\bigr)^{-1},
      \end{equation*}
      and so $\|R(\lambda,W+\widetilde W)\|_{\mathcal L(\mathcal H_s,\mathcal H_{-s})}$ is uniformly bounded with respect to $\widetilde W$, for
      $\widetilde W$ small. This proves \eqref{E:Q-der-W}.
      By induction in \eqref{E:differentiable} it follows that $R$ is $C^\infty$ in the $W$ variable, and that
      $R^{(j)}_W(\lambda,W)$ is a multilinear map such that
      $R^{(j)}_W(\lambda,W)(\widetilde W_1,\dots,\widetilde W_j)$ is
      of the form $M(\lambda,W;\widetilde W_1,\dots,\widetilde W_j)$, where $M(\lambda,W;\widetilde W_1,\dots,\widetilde W_j)$ is a sum of products with
      $2j +1$ factors where every second factor is $(\overline H + W - \lambda - i0)^{-1}$ and every second factor is $\widetilde W_l$ for some
      $l\in \{1,\dots,j\}$.

      By Proposition \ref{Prop:uniform2} the derivatives
      $\displaystyle{\frac{\partial^j}{\partial \lambda^j}} R(\lambda,W)$ exist for $j\le k$,
      and  since the terms of $R^{(j)}_W(\lambda,W;\widetilde W_1,\dots,\widetilde W_j)$ are  compositions of resolvents and
      $\widetilde W_l$, it follows from the product rule that the mixed derivatives exist and are continuous up to order $k$ in
      $\lambda$ when we apply the $W$-derivatives first and then the $\lambda$-derivatives.

      To prove that the partial derivatives taken in another order exist and are continuous, we will use a corresponding result from calculus.
      Let $f$ be a function of $x_1,\dots, x_n$ such that all the mixed partial derivatives up to order $m$ exist and are continuous when the
      partial derivatives are taken in the order of increasing index of the variables.
      Hence, we assume that
      $\partial^r f/\partial x_{l_1}\dots\partial x_{l_r}$  exist and are continuous for every $r\le m$ and every $l_1,\dots, l_r$ such that $l_1 \le l_2\le \dots \le l_r$.
      Then $f\in C^m(\mathbb R^n)$. The proof when $m=2$ follows from Theorem 1 of \cite[p.163]{lB55}, and  the
      general case follows by induction on the order of the derivative.

      Let $1\le m\le k$. The $m$th Gateaux derivative  of $R$ (if it exists) is given by
      \begin{equation}\label{E:Gateaux}
         \frac{\partial}{\partial t_1}\dots\frac{\partial}{\partial t_m} \left. R\biggl((\lambda,W) + \sum_{l=1}^m t_l
         (\widetilde \lambda_l,\widetilde W_l))\biggr)\right|_{
         \begin{smallmatrix}
            t_l=0 \\
            l=1,\dots,m
         \end{smallmatrix}}.
      \end{equation}
      Let $0\le r\le m$ be arbitrary.
      By choosing $\widetilde \lambda_l=0$ for $l=0,\dots, r$ and $\widetilde W_l=0$ for $l=r+1,\dots,m$, we get from \eqref{E:Gateaux} the mixed
      partial derivative where we first differentiate $l-r$ times with respect to $W$ and then $r$ times with respect to $\lambda$. This is
      the derivative which we know exists and is continuous for $|\lambda-\lambda_0|$, $\|W\|_{X}$, $s_l$, $t_l$ small.
      By the calculus result quoted above, we may change the order of
      differentiation in \eqref{E:Gateaux}, and we see that all
      the mixed partial derivatives of total order $m$ of $R$ exist and are continuous. To conclude that $R$ is $k$ times Gateaux
      differentiable,
      we let $g$ be the function
      \begin{equation*}
          g(s_1,\dots,s_m,t_1,\dots,t_m) := R\biggl((\lambda,W) + \sum_{l=1}^m \left(s_l (\widetilde \lambda_l,0) + t_l
          (0,\widetilde W_l)\right)\biggr),
      \end{equation*}
      for an arbitrary choice of $\widetilde \lambda_l$ and $\widetilde W_l$, $l=1,\dots,m$.
      Note that $g$ has continuous partial derivatives of order $m$. Let
      $h(t_1,\dots,t_m):= g(t_1,\dots,t_m,t_1,\dots,t_m)$ and note that by \eqref{E:Gateaux} that the $m$th order Gateaux derivative of $R$ is just a
      mixed partial derivative of $h$.
      We obtain from the chain rule that
      $h$ is $m$ times continuously differentiable, and so
      \eqref{E:Gateaux} holds for any choice of
      $\widetilde \lambda_l$,  $\widetilde W_l$, $l=1,\dots,m$, and since $m\in\{1,\dots,k\}$ was arbitrary we see that all the Gateaux
      derivatives  up to order $k$ are continuous with respect to $\lambda$ and $W$ in a neighborhood of $(\lambda_0,0)$ and multilinear.
      By \cite[p. 73]{mB77}, $R$ is also $k$ times continuously Fr\'echet differentiable in this neighborhood.
   \end{proof}

   \section{The equation $Q(\lambda + i0,W)f = f$}\label{S:Q}\noindent
   Proposition \ref{P:ev} leads us to the study of the equation
   \begin{equation}\label{E:Q-eq}
      Q(\lambda + i 0,W)f=f.
   \end{equation}
   If \eqref{E:Q-eq} holds, then
   \begin{equation*}
      \langle f,f\rangle = \langle Q(\lambda + i0, W) f, f\rangle = \langle f,Q(\lambda - i0, W)f\rangle = \langle Q(\lambda - i0,W) f, f\rangle.
   \end{equation*}
    By \eqref{E:delta-def}
    it follows that $\langle \delta (\overline H + W - \lambda)f,f\rangle = 0$.

   Note that $\langle \delta(\overline H + W - \lambda) f, g\rangle$ defines a sesquilinear form on $\mathcal H_s$, for which the Schwarz inequality holds.
   Hence, for every $g\in \mathcal H_s$
   \begin{equation*}
      |\langle \delta(\overline H + W - \lambda) f,g\rangle|^2 \le \langle \delta(\overline H + W - \lambda)f,f\rangle
      \langle \delta(\overline H + W - \lambda) g, g\rangle =0.
   \end{equation*}
   It  follows that  \eqref{E:Q-eq} is equivalent to
   \begin{equation}\label{E:deltaA-equation}
      \begin{aligned}
         \delta(\overline H + W - \lambda) f &= 0, \\
         A(\lambda,W)f&=f,
      \end{aligned}
   \end{equation}
   where $A(\lambda,W)$ is given by \eqref{E:A-def}.

   We first study the second equation of \eqref{E:deltaA-equation}
   for  $f\in \text{Ran }P_0$. We focus on the non-degenerate case, i.e. we assume that $\lambda_0$ has multiplicity $1$.
   \begin{Prop}\label{P:A-equation}
      Suppose that (H1), (H2), and $\text{(H3)}_k$ with $k \ge 1$ are satisfied, and that the eigenvalue $\lambda_0$  has multiplicity $1$ . Suppose $s > s_1$.  Then the second equation of \eqref{E:deltaA-equation} defines $\lambda = \lambda(W)$ in a neighborhood of
      $(\lambda,W) = (\lambda_0,0)\in \mathbb R\times X_s$. Moreover, $\lambda(\cdot)$ is a $C^k$ function and
      $\lambda'(0)\widetilde W = \langle \varphi_0,\widetilde W\varphi_0\rangle$.
   \end{Prop}
   \begin{proof}
      It is  natural to identify
      the operator $A(\lambda, W)$ with the function $\langle \varphi_0,A(\lambda,W)\varphi_0\rangle$, where
      $P_0 \varphi_0=\varphi_0$ and $\|\varphi_0\|_{\mathcal H}=1$. We then have
      \begin{equation}\label{E:A-formula}
         A(\lambda,W) = \frac{1}{2} \langle \varphi_0,\left((\overline H + W - \lambda - i0)^{-1} + (\overline H + W - \lambda + i0)^{-1}\right) \varphi_0\rangle.
      \end{equation}
      Since by Proposition \ref{P:differentiable},
      $Q\in C^k(J\times \mathcal O;\mathbb C)$ where $J\times \mathcal O$ is a neighborhood of $(\lambda_0, 0)$ in $\mathbb R \times X_s$ it follows that also $A\in C^k(J\times \mathcal O;\mathbb C)$. By self-adjointness of $H$ and $W$, we have
      for every $\epsilon>0$ that
      \begin{equation*}
         \left((\overline H + W - \lambda - i\epsilon)^{-1}\right)^* = (\overline H + W - \lambda + i\epsilon)^{-1}.
      \end{equation*}
      It follows that
      $A(\lambda,W) = \overline{A(\lambda,W)}$, and so
      $A\in C^k(J\times \mathcal O;\mathbb R)$.

      By \eqref{E:A-formula} and since $\varphi_0$ is an eigenvector of
      $H$ with eigenvalue $\lambda_0$,
      \begin{equation*}
         A(\lambda,0) = \frac{1}{\lambda_0 + 1 - \lambda} =: c(\lambda).
      \end{equation*}
      Observing that
      $c'(\lambda)=1/(\lambda_0 + 1 - \lambda)^2$, we see that
      \begin{equation*}
          A'_\lambda(\lambda_0,0)=1,
      \end{equation*}
      and $A(\lambda_0,0)=1$. By the implicit function theorem $A(\lambda,W)=1$ defines $\lambda$ as a $C^k$ function of $W$ in a neighborhood of
      $\lambda=\lambda_0$ and for $W\in X_s$ small.

      By \eqref{E:Q-der-W},
      and since $\varphi_0$ is an eigenvector of $H$ with eigenvalue $\lambda_0$,
      we obtain
      \begin{equation*}
         A'_W(\lambda_0,0) W = -\langle \varphi_0,W\varphi_0\rangle.
      \end{equation*}
      Since $A'_\lambda(\lambda_0,0) = 1$ it follows that
      $\lambda'(0)W = \langle \varphi_0,W\varphi_0\rangle$.
   \end{proof}

  \begin{Prop}\label{P:pert}
   Suppose that (H1), (H2) and $\text{(H3)}_0$ are satisfied and $J_1$, $\gamma$, $s_1$, and $\delta_1$ are as in Proposition \ref{Prop:uniform2}.  Then if $s > s_1$, $\|W\|_{\mathcal  L(\mathcal H_{-s},\mathcal H_{s})} < \gamma$,  and $\lambda\in J_1$,
   we then have the following perturbation formula for $\delta(\overline H + W - \lambda)$:
      \begin{equation*}
         \delta(\overline H + W - \lambda) = \left( I - (\overline H + W - \lambda - i 0)^{-1}W\right)\delta(\overline H
         - \lambda)\left(I - W(\overline H + W - \lambda+ i 0)^{-1}\right).
      \end{equation*}
   \end{Prop}
   \begin{proof}
      We have
      \begin{equation*}
         \begin{aligned}
            (\overline H + W - \lambda-i\epsilon)^{-1} &= \left(I - (\overline H  + W - \lambda-i\epsilon)^{-1} W\right)(\overline H - \lambda-i\epsilon)^{-1}, \\
            (\overline H + W - \lambda+i\epsilon)^{-1} &= \left(\overline H - \lambda+i\epsilon\right)^{-1}\left(I - W(\overline H + W - \lambda+i\epsilon\right)^{-1}),
         \end{aligned}
      \end{equation*}
      which imply
      \begin{equation*}
         \begin{aligned}
            &(\overline H + W - \lambda-i\epsilon)^{-1} \frac{\epsilon}{\pi}(\overline H + W - \lambda+i\epsilon)^{-1} \\
            &= \left( I - (\overline H + W - \lambda-i\epsilon)^{-1} W\right)(\overline H -
            \lambda-i\epsilon)^{-1} \frac{\epsilon}{\pi} (\overline H - \lambda+i\epsilon)^{-1} \left(I - W(\overline H + W - \lambda+i\epsilon)^{-1}\right).
         \end{aligned}
      \end{equation*}
      By \eqref{E:delta-def}, this proves the proposition.
   \end{proof}

\section{Finite multiplicity of the continuous spectrum}\label{S:finmult}\noindent
\begin{proof}[Proof of Theorem \ref{T:main}]
By $\text{(H3)}_k$, $\Ran \delta(\overline H - \lambda)\subset \mathcal H_{-s}$ for $s>s_1$. By (H4), $\Ran \delta(\overline H - \lambda_0)$ is $m$-dimensional. Let
$\varphi_1, \dots, \varphi_m\in \mathcal H_s$  and $f_1,\dots,f_m\in \mathcal H_{-s}$ be linearly independent and satisfy
\begin{equation}\label{E:varphi_j f_j}
   \delta(\overline H - \lambda_0)\varphi_j = f_j.
\end{equation}
We may without loss of generality assume that $f_j$ for $j=1,\dots, m$ are real. Indeed, suppose that there are only $j\le m-1$ real linearly independent vectors
$f_1,\dots,f_j\in \Ran \delta(\overline H - \lambda_0)$. Since $\Ran \delta(\overline H - \lambda_0)$ is $m$-dimensional, we can choose
$f\in \Ran \delta(\overline H - \lambda_0)$ such that $f_1,\dots,f_j,f$ are linearly independent. Let $\Re f= (f + Cf)/2$ and $\Im f=(f-Cf)/2i$ so that $f = \Re f + i \Im f$.
It is not possible that both $\Im f$ and $\Re f$  are linear combinations of $f_1,\dots,f_j$, since if they are then so is $f$. Hence one of $\Im f$ and $\Re f$ is not a linear
combination of $f_1,\dots, f_j$, say $\Re f$. But then there are $j+1$ real linearly independent vectors that span $\Ran \delta(\overline H - \lambda_0)$, contradicting our
assumption that only $j$ such vectors exist.

For $W$ in a sufficiently small neighborhood of $0\in X_s$, let
\begin{equation*}
   f_j(W) :=  \delta(\overline H + W - \lambda(W))\varphi_j,
\end{equation*}
where $\lambda(W)$ is defined as in Proposition \ref{P:A-equation}. Note that $f_j(0) = f_j$. By Proposition
\ref{P:differentiable}, $f_j(\cdot)\in C^k(X_s;\mathcal H_{-s})$. Note
 that
$(I-W(\overline H + W - \lambda + i0)^{-1}): \mathcal H_s\to \mathcal H_s$ and $(I-(\overline H + W - \lambda -i0)^{-1}W):\mathcal H_{-s}\to\mathcal H_{-s}$ are invertible. Indeed, the inverses are given by $(I + W(\overline H -\lambda + i0)^{-1})$ and $(I + (\overline H - \lambda - i0)^{-1})$, respectively. Then by Proposition \ref{P:pert} and (H4), $\{f_j(W);j=1,\dots,
m\}$ span the $m$-dimensional subspace $\Ran \delta(\overline H + W - \lambda(W))\subset \mathcal H_{-s}$ if
$\|W\|_{\mathcal L(\mathcal H_{-s},\mathcal H_s)}$ is sufficiently small. Let $g_l\in \mathcal H_s$, $l=1,\dots,m$ be
such that
\begin{equation}\label{E:dual basis}
   \langle f_j(0),g_l\rangle = \delta_{jl}
\end{equation}
for $j, l\in \{1,\dots, m\}$. Note that we may assume that also the $g_l$ are real, since
\begin{equation*}
   \langle f_j(0),C g_l\rangle = \langle C f_j,C g_l\rangle = \overline{\langle f_j, g_l\rangle} = \overline{\delta_{jl}} = \delta_{jl}.
\end{equation*}
Hence we may replace $g_l$ by $(g_l + C g_l)/2$.

We claim that for $W\in X_s$ small, the equation $\delta(\overline H + W - \lambda(W))\varphi_0=0$ is equivalent to
\begin{equation}\label{E:F_j-eqn}
  F_j(W) := \langle g_j, \delta(\overline H + W - \lambda(W))\varphi_0\rangle = 0,
\end{equation}
$j=1,\dots,m$. To verify this, it suffices to show that \eqref{E:F_j-eqn} implies that $\delta(\overline H + W - \lambda(W))\varphi_0=0$ for
$\|W\|_{\mathcal L(\mathcal H_{-s},\mathcal H_s)}$ small, since the other implication is trivial. Write
$\delta(\overline H + W - \lambda(W))\varphi_0 = \sum_{l=1}^m \alpha_l(W) f_l(W)$, and suppose that \eqref{E:F_j-eqn} holds. Then for every $j\in\{1,\dots, m\}$
\begin{equation}\label{E:alpha-eq}
   \sum_{l=1}^m \alpha_l(W) \langle g_j, f_l(W)\rangle =0.
\end{equation}
Note that the $m\times m$-matrix with entries $\langle g_j, f_l(W)\rangle$ is continuous and equal to the identity matrix when $W=0$. Hence it
is invertible for $\|W\|_{\mathcal L(\mathcal H_{-s},\mathcal H_s)}$ small, and from \eqref{E:alpha-eq} we obtain $\alpha_j(W)=0$ for  $\|W\|_{\mathcal L(\mathcal H_{-s},\mathcal H_s)}$ small.

By Proposition \ref{P:differentiable} we have for some neighborhood $\mathcal O$ of $0 \in X_s$, $F\in C^k(\mathcal O;\mathbb C^m)$, where $F_j$ are the components of $F$. Note
that $\varphi_0$ can be chosen real since otherwise we may replace $\varphi_0$ by its real or imaginary parts (i.e. $(\varphi_0 + C \varphi_0)/2$ or $(\varphi_0 -C\varphi_0)/2$). From
our choice of $g_j$ it now follows that  $F\in C^k(\mathcal O;\mathbb R^m)$.

We need to calculate $F_j'(0) W$. Note that $\delta(\overline H - \lambda)\varphi_0 = 0$ for every real $\lambda$ in a neighborhood of $\lambda_0$, and that $(\overline H - \lambda_0 + i0)^{-1} \varphi_0 = \varphi_0$,
and so it follows from  the chain rule, Proposition \ref{P:differentiable} and Proposition \ref{P:pert} that for $\lambda = \lambda_0$
\begin{equation*}
   \begin{aligned}
      F_j'(0) W &= -\langle g_j,\delta(\overline H - \lambda) W (\overline H - \lambda + i0)^{-1} \varphi_0\rangle
      - \langle g_j, (\overline H - \lambda - i0)^{-1} W \delta(\overline H - \lambda) \varphi_0 \rangle \\
      &\qquad \qquad +  \langle \varphi_0,W\varphi_0\rangle \langle g_j, \frac{d}{d\lambda}
       \delta(\overline H - \lambda)\varphi_0\rangle \\
&= -\langle g_j, \delta(\overline H - \lambda_0) W \varphi_0 \rangle.
   \end{aligned}
\end{equation*}
We need to show that $F_1'(0),\dots,F_m'(0)$ are linearly independent. To see this, let
\begin{equation*}
   \sum_{j=1}^m \alpha_j F'_j(0) = 0,
\end{equation*}
and let $g:= \sum_{j=1}^m \alpha_j g_j$. Then for every $W\in X_s$
\begin{equation*}
   \langle  g,\delta(\overline H -\lambda_0) W \varphi_0\rangle = 0.
\end{equation*}
By (H5)  and by the definition of $g$ it follows that $g=0$, and so by the linear independence of $g_1,\dots,g_m$, we obtain that $\alpha_j = 0$
for $j=1,\dots,m$, and we conclude that  $F_1'(0),\dots,F_m'(0)$ are linearly independent.

We are now able to make the decomposition $X_s = \ker F'(0) \oplus \mathcal M$, where $\mathcal M$ has dimension $m$.
Moreover, the map $F'(0):\mathcal M\to \mathbb R^m$ is a linear homeomorphism and $F(0)=0$. For $W\in X_s$, we write $W =
\xi + \eta$ where $\xi\in \ker F'(0)$ and $\eta\in \mathcal M$.  We also use the notation $F(\xi,\eta) = F(\xi +
\eta)$. By the implicit function theorem the equation $F(W)=0$ can be solved for $\eta$ in terms of $\xi$, i.e. $\eta =
\eta(\xi)$ in a neighborhood of $0$, and for some neighborhood $U$ of $0 \in \ker F'(0), \eta\in C^k(U;\mathcal M)$. This defines a $C^k$ manifold of
codimension $m$ in a neighborhood of $0$.
\end{proof}

\section{Applications to elliptic differential operators}\label{S:applications}\noindent
Here we present some examples for which the assumptions (H1)--(H5) can be verified.

\begin{Ex}\label{Ex:line}
   Let $\mathcal H:= L^2(\mathbb R)$, and let $\mathcal H_s := L^2_s(\mathbb R)$, where $L^2_s(\mathbb R)$ is the Hilbert space of functions $\psi$ such
   that $(1+x^2)^{s/2} \psi(x)$ is square integrable. Let  $H:=d^4/dx^4 + V(x)$, where $V\in C^{k+2}(\mathbb R)$
   is a real  potential  satisfying
   \begin{enumerate}
      \item[(i)] $\sup_{x\in \mathbb R}|V^{(j)}(x)|(1+|x|^2)^{j/2}<\infty$ for $j=0,\dots, k+2$,
      \item[(ii)] $\sup_{x\in \mathbb R} |V(x)| (1 + |x|^2)^{q} <\infty$, where $q>1/2$.
   \end{enumerate}
   We consider $H$ as an unbounded operator in $L^2(\mathbb R)$ with its domain being the Sobolev space $W^{4,2}(\mathbb R)$ (also denoted by $H^4(\mathbb R)$), where $H u =\displaystyle{\frac{d^4 u}{dx^4}} + V u$ for $u\in \Dom(H)$. Denote by $H_0$ the same operator with $V\equiv 0$. It is readily seen (by Fourier transform) that $H_0$ is a self-adjoint operator with $\sigma(H_0)=\sigma_c(H_0)=\mathbb R_+$. Since $V$ is a bounded function on $\mathbb R$ which tends to $0$ as $|x|\to\infty$, it follows by well known results that
   $H$ is a self-adjoint operator and that $\sigma_c(H) = \sigma_c(H_0)=\mathbb R_+$. Thus in particular $H$ satisfies assumption (H1).
   We shall assume now that
   \begin{enumerate}
      \item[(iii)]{$H$ has an embedded eigenvalue $\lambda_0>0$ with multiplicity $1$.\footnote{That this can be achieved can be seen
      for the example when $V$ is  given by
      $V(x) = 20/\cosh^2 x - 24/\cosh^4 x$. A short calculation
      shows that $\lambda_0=1$ is an embedded eigenvalue and that the corresponding eigenfunction is $\varphi_0(x)= 1/\cosh x$.
      It follows from ODE theory \cite{CL55} that $\lambda_0$ is a simple eigenvalue, since the equation $d^4 \psi/dx^4 + V(x) \psi =
      \lambda_0 \psi$ has exactly one linearly independent solution which decays as $x\to+\infty$ (or $x\to -\infty$).}}
   \end{enumerate}
   We take $s > k + 1/2$ and $k \ge 1$.  As the space of perturbations $X_s$, we choose the set of real multiplication operators in $\mathcal L(\mathcal H_{-s},\mathcal H_s)$, i.e. multiplication by real functions $W$ on $\mathbb R$ which satisfy
   $$\sup_{x\in \mathbb R} (1 + |x|^2)^{s}|W(x)|<\infty.$$
   As the antiunitary involution we choose complex conjugation. Below we show that the conditions (H2)--(H5) are satisfied for this operator
   with $s_1=k+1/2$ in condition $\text{(H3)}_k$ and $m=2$ in condition (H4). Theorem \ref{T:main} then implies that the set of small perturbations which do not remove the embedded eigenvalue is a $C^k$ manifold in $X_s$ of codimension $2$.

   To verify (H2), we use Theorem 4.1 and inequality (4.1) of \cite{mBA79} or Theorem 30.2.10 of \cite{lH85},  which shows that the eigenvalues of $H$ can only accumulate at
   $0$ and that the eigenfunctions belong to $\mathcal H_s$ for every $s$.  Theorem 30.2.10 of \cite{lH85} shows that  $\text{(H3)}_0$ is satisfied for $s>1/2$.\footnote{(H2) and $\text{(H3)}_0$ could also be verified by the methods of \cite{sA75}.}

   To prove $\text{(H3)}_k$, we need to verify the assumptions of Theorem 2.2. of \cite{JMP84} with
   $A = \frac{1}{8}(x D + D x)$, where $D:= -i d/dx$. The calculations are similar to those in Section I of \cite{eM81} and Section 5 of \cite{JMP84}, but we include them here for the convenience of the reader.

   Note that $\mathscr S(\mathbb R)$ is
   a common core for $\overline H$ and $A$, and so we may compute the commutators on $\mathscr S(\mathbb R)$. Hence $(a)$ of Definition 2.1 in
   \cite{JMP84} is satisfied, i.e. $\Dom(A) \cap \Dom(\overline H)$ is a core for $\overline H$.

   Condition $(b)$ of \cite{JMP84} is that $e^{i\theta A}$ maps $\Dom(\overline H)$ into $\Dom(\overline H)$ and for
   each $\psi\in \Dom(\overline H)$
   \begin{equation}\label{E:semigroup-estimate}
      \sup_{|\theta|\le 1} \| \overline H e^{i\theta A} \psi\|<\infty.
   \end{equation}
   To prove this, we use the formula
   \begin{equation}\label{E:semigroup-formula}
       e^{i\theta A}f(x) = e^{\theta/8} f(e^{\theta/4}x),
   \end{equation}
   which holds since the left and right hand sides of \eqref{E:semigroup-formula} define $C_0$ semigroups with the same infinitesimal generator $iA$. By the Hille--Yosida Theorem, the semigroups must be equal.
   By using \eqref{E:semigroup-formula}, it is easy to see that $e^{i\theta A}$ maps $\Dom(\overline H)$ into $\Dom(\overline H)$ and that \eqref{E:semigroup-estimate} holds.

   Let $B_0=\overline H$. The condition $(c_{k+1})$ of \cite{JMP84} requires that the forms $i^j B_j$ defined on
   $\Dom(\overline H)\cap \Dom(A)$ are all bounded
   from below and closable, and that $\Dom(B_j)\supset \Dom \overline H$, where  $B_j$ is the closure of $[B_{j-1},A]$ for
   $j=1,\dots, k+1$, and
   the commutator $[B_{j-1},A]$ is interpreted as a quadratic form, i.e.
   \begin{equation*}
      \langle\varphi,[B_{j-1},A]\psi\rangle  := \langle B_{j-1}\varphi,A\psi\rangle - \langle A\varphi,B_{j-1}\psi\rangle,
   \end{equation*}
   for $\varphi, \psi\in \Dom(A)\cap \Dom(\overline H)$. To verify this, we first use Theorem 8.1 of \cite{CL55} to show
   that $\varphi_0$ and its derivatives are exponentially decaying as $|x|\to\infty$. Indeed, after rewriting the
   eigenvalue equation as a system of four linear ODE's in the standard way, this theorem implies that
   $\varphi_0$, $\varphi_0'$, $\varphi_0''$ and
   $\varphi_0^{(3)}$ are all exponentially decaying. From the eigenvalue equation
   $\varphi_0^{(4)}=\lambda_0\varphi_0-V\varphi_0$ and by (i), it follows that $\varphi_0^{(4)}$ is exponentially decaying. We now differentiate
   this equation and proceed by induction. We see that $\varphi_0^{(j)}$ is exponentially decaying for $j=1,\dots, k+6$. It also follows that $A^j \varphi_0$ is exponentially decaying for each $j\le k+6$. In particular, $\varphi_0\in \Dom(A^j)$
   for $j=1,\dots,k+6$ and $A^j P_0$ is defined for those $j$.
   A calculation shows that
   \begin{equation}\label{E:commutator formula}
      i^j B_j = \frac{d^4}{dx^4} + \frac{(-1)^j}{4^j}  \left(x\frac{d}{dx}\right)^j V(x) + i^j \sum_{l=0}^j (-1)^l\binom{j}{l}
      A^l P_0 A^{j-l},
   \end{equation}
   is bounded from below and closable when $j\le k+1$, and its closure (also denoted by $B_j$) has the domain
   $\Dom{(B_j)} = \Dom(H) =\Dom{(\overline H)}$.
   Hence $(c_{k+1})$  of \cite{JMP84} is satisfied.

   Condition $(d_{k+1})$ of \cite{JMP84} states that the form $[B_{k+1},A]$ defined on $\Dom(\overline H)\cap \Dom(A)$ extends to a bounded operator from $\Dom(\overline{H})$ equipped with the graph norm to its dual obtained by the inner product on $\mathcal H$.   Using \eqref{E:commutator formula}, this is straightforward to check, since $B_{k+2}$
   is a bounded operator from  $\Dom(\overline H)$ to $L^2(\mathbb R)$.

   Finally, we verify the Mourre estimate $(e)$ of \cite{JMP84}, i.e. we need to verify that there exist $\alpha>0$, $\delta>0$, and a
   compact operator $K$ on $\mathcal H$ such that
   \begin{equation*}
       E_{J}(\overline H) i B_1 E_{J}(\overline H) \ge \alpha E_{J}(\overline H) + E_{J}(\overline H) K
       E_{J}(\overline H),
   \end{equation*}
   where $J:= (\lambda_0-\delta,\lambda_0+\delta)$. Let $0<\delta<\lambda_0$,
   and let $K:= (-V + i[V,A]-P_0+i[P_0,A]) E_{J}(\overline{H})$. The assumption (ii) on $V$ ensures that
   $(-V + i[V,A]-P_0+i[P_0,A])$ is $\overline H$-compact, and
   hence  $K$ is compact. By \eqref{E:commutator formula} for $j=1$,
   \begin{equation*}
      \begin{aligned}
         E_{J}(\overline H) i B_1 E_{J}(\overline H) &= E_{J}(\overline H) (\overline H - V + i[V,A]-P_0+i[P_0,A]) E_{\overline H}(J) \\
         &= E_{J}(\overline H) \overline H E_{J}(\overline H) + E_{J}(\overline H) K E_{J}(\overline H) \\
         &\ge(\lambda_0-\delta) E_{J}(\overline H) + E_{J}(\overline H) K E_{J}(\overline H).
      \end{aligned}
   \end{equation*}
   According to \cite{JMP84} this shows that the limits
   $$\lim_{\epsilon \downarrow 0}(1+A^2)^{-s/2}(\overline H - \lambda \pm i\epsilon)^{-1}(1+A^2)^{-s/2} = (1+A^2)^{-s/2}(\overline H - \lambda \pm i0)^{-1}(1+A^2)^{-s/2}$$
   \text {exist in} \,$\mathcal  L(\mathcal H )$ and are $C^{k}$
   in $\lambda$ in the norm topology of $\mathcal  L(\mathcal H )$ for
   $\lambda$ in some interval around $\lambda_0$.  We must now prove that $A$ can be replaced by $x$ in the latter statement.  It suffices to take $s=k+1$.  Using the resolvent equation repeatedly we see that it is enough to show that $A^s (\overline H + N)^{-n}(1+x^2)^{-s/2}$ is bounded for some large $N$ and $n$. By interpolation it suffices to take $s=k+1$. Since $A^s P_0$ is bounded it is enough to show boundedness of $D^lx^l(H + N)^{-n}(1+x^2)^{-s/2}$, where $D = -id/dx$ and $l\le s$.  Let $R = (H + N)^{-1}$.  We will control the terms generated by taking commutators with $x^l$ by the following lemma, which is easily proved by induction.
   \begin{Lem}\label{L:Sobolev}
   Suppose that $l \ge 0 $, $k \in \{0,1,2,3,4\}$, and  $V \in C^{(l + k -4)_+}(\mathbb{R})$ with bounded derivatives. Then  $$R : \mathcal H^l(\mathbb{R}) \rightarrow \mathcal H^{l+k}(\mathbb{R})$$ is bounded.
   \end{Lem}

   Commuting $x^l$ through $R^n$ produces terms with factors of $n$ resolvents interspersed with ($\le s$) factors of the form $D^jR$ where $j = 0, 1, 2$ or $3$.  The string of factors always begins with $R$ on the left.  Let us use the new factors $D^j R$ to map $H^l \rightarrow H^l$ and the old factors $R$ to increase the Sobolev index to $s$ which we write as $s = 4k_0 + m$ where
   $m = 0, 1, 2$, or $3$.  We will then know that $D^s$ times the operator string is bounded.  The only question is how many bounded derivatives of $V$ this requires.  Suppose after applying a string including $r$ of the original $R$'s and any number of the new $D^jR$'s to $L^2(\mathbb R)$ we find ourselves in $H^{4r}$ needing at most $4r-1$ bounded derivatives of $V$.  According to Lemma \ref{L:Sobolev}, applying $r'$ additional $R$'s brings us to $H^{4r+4r'}$ needing $4r + 4r' - 4$ bounded derivatives of $V$.  Applying any number of $D^jR$'s requires at most $4r + 4r' -1$ bounded derivatives to stay in $H^{4r+4r'}$.  Thus inductively we can reach $H^{4k_0}$ needing at most $4k_0 -1$ bounded derivatives.  We use the last $R$ to reach $H^{4k_0 + m}$ with no further derivatives needed.  Thus the requirement that $A^s (\overline H + N)^{-n}(1+x^2)^{-s/2}$ is bounded requires at most $s-1 = k$ bounded derivatives which we have by assumption (i).

   This concludes the proof of $\text{(H3)}_k$.

   We proceed by verifying (H4). More precisely, we will check that $\dim \Ran \delta(\overline H-\lambda)=2$ for $\lambda$ in a neighborhood of $\lambda_0$. Let $H_0:= d^4/dx^4$. It is clear that $\dim\Ran(H_0-\lambda)=2$ if $\lambda>0$. Indeed, the
   range is the span of the functions $e^{i\lambda^{1/4}\cdot}$ and $e^{-i\lambda^{1/4}\cdot}$.
   We now apply Proposition \ref{P:pert} with $H_0$ taking the place of $H$ and $P_0+V$ taking the place of $W$. We then get
   \begin{equation*}
      \delta(\overline H - \lambda) = (I - (\overline H-\lambda- i0)^{-1}(P_0+V))\delta(H_0-\lambda) (I-(P_0+V)(\overline H - \lambda +i0)^{-1}).
   \end{equation*}
   Since $(I-(P_0+V)(\overline H - \lambda +i0)^{-1})$ and $(I - (\overline H-\lambda- i0)^{-1}(P_0+V))$ are invertible (with inverses $(I+(P_0+V)(H_0-\lambda+i0)^{-1})$ and $(I + (H_0-\lambda -i0)^{-1}(P_0+V))$, respectively), this gives a one-to-one correspondence between the range of $\delta(H_0+V-\lambda)$ and the range of $\delta(\overline H +V-\lambda)$. We conclude that the dimensions of the ranges must be equal.

   (H5) is satisfied since  $\{W \varphi_0;\; W\in X_s\}$ is dense in $\mathcal H_s$, which follows since the zero set of $\varphi_0$ is at most
   countable with no accumulation points.
\end{Ex}
\begin{Ex}\label{Ex:cylinder}
   Let $M$ be the infinite cylinder $\mathbb R \times S^1$ with generic point $x=(z,\theta)$. $M$ is considered as a Riemannian manifold  with the metric $dx^2 = dz^2 + d\theta^2$. Let $H$ be a Schr\"odinger operator on $M$ of the form
   \begin{equation*}
      H := -\Delta + V= -\left(\frac{d^2}{dz^2} + \frac{d^2}{d\theta^2}\right) + V(z),
   \end{equation*}
   where $V$ is a real function in $C_0^\infty(\mathbb R)$, and consider $H$ as an unbounded operator in $L^2(M)$ with the domain being the Sobolev space $W^{2,2}(\mathbb R\times S^1)$. Viewing $H$ as a perturbation of the operator $H_0:=-\Delta$, it is easy to see (as in Example \ref{Ex:line}) that $H$ is a self-adjoint operator and that $\sigma_c(H) = \sigma_c(H_0)=\mathbb R_+$. Now, we shall choose the potential $V\in C_0^\infty(\mathbb R)$ so that $h:= -d^2/dz^2 + V$ on $L^2(\mathbb R)$ has exactly one eigenvalue $e<0$ of multiplicity $1$.
   It is possible to choose $V$ such that $e>-1$. We denote the corresponding eigenfunction by $f$. Let $n\ge 1$ and let $F(z,\theta) =
   \cos(n \theta) f(z)$. Then $H F = \lambda_0 F$ with $\lambda_0=n^2 + e>0$, and so $\lambda_0$ is an embedded eigenvalue.
   Note that the
   multiplicity of $\lambda_0$ is $2$ since $\sin(n \theta)f(z)$ is also an eigenfunction. This degeneracy can be removed by
   letting $\mathcal H$ be the subspace of $L^2(M)$ consisting of functions which are even in the $\theta$-variable. Let
   \begin{equation*}
      \mathcal H_s = \{f\in \mathcal H;\; (1 + z^2)^{s/2}f\in L^2(M)\},
   \end{equation*}
   and let $X_s$ be the space of real
   multiplication operators $W$ which are even in $\theta$ and satisfy
   \begin{equation*}
      \sup_{
      \begin{smallmatrix}
         z\in\mathbb R, \\
         \theta\in S^1
      \end{smallmatrix}}
      (1 + |z|^2)^s |W(z,\theta)| <\infty.
   \end{equation*}
   The antiunitary involution is complex conjugation.
   Below we show that conditions (H1)--(H4) are satisfied for this operator for any fixed $k\ge 1$ with $s_1=k+1/2$ in condition $\text{(H3)}_k$ and $m=2n$ in condition (H4). Then Theorem \ref{T:main} tells us that the set of small perturbations for which the embedded eigenvalue persists is a $C^k$ manifold in $X_s$ of codimension $2n$.

   Assumptions (H2) and $\text{(H3)}_0$ are verified (via separation of variables) using, as in Example \ref{Ex:line}, well known limiting absorption results for the Schr\"odinger operator on $\mathbb R$ (e.g. \cite{sA75, lH85}). We omit the details.

   To verify $(H3)_k$, we apply the result from \cite{JMP84} with $A=(zD+Dz)/4$, where
   $D = -i \partial/\partial z$. Note that $i[H_0,A]=D^2$, and let $J=(\lambda_0-\delta,\lambda_0+\delta)$, where
   $\delta<\min(1,\min_{j\in\mathbb Z}|\lambda_0-j^2|)=:\alpha$.
   Let $E_J(\overline H)$ be the spectral projection of $\overline H$ in $L^2(\mathbb R\times S^1)$ corresponding to the interval $J$, and let $1_J(h)$ be the spectral projection of $h$ in $L^2(\mathbb R)$ corresponding to $J$. Let $Q_n=1_{\{n^2\}}(-\partial^2_\theta)$
   in $L^2(S^1)$.
   Note that
   \begin{equation*}
      E_J(\overline H) = E_J(\overline H) P_0 + E_J(\overline H)(I-P_0) = E_J(H)(I-P_0),
   \end{equation*}
   since $J\cap \{\lambda_0+1\}=\emptyset$.
   By the choice of $J$,
   \begin{equation*}
      J\cap \sigma(h+j^2) = J\cap ([j^2,\infty)\cup \{e+j^2\}) =
      \begin{cases}
         \{\lambda_0\} &\text{if }j^2=n^2,  \\
         \emptyset &\text{if }j^2>\lambda_0\text{ and }j^2\ne n^2, \\
         J &\text{if }j^2<\lambda_0,
      \end{cases}
   \end{equation*}
   where $\sigma(h+j^2)$ is the spectrum of $h+j^2$ in $L^2(\mathbb R)$.
   Note that
   \begin{equation*}
       P_0 = 1_{\{\lambda_0\}}(h+n^2)\otimes Q_n,
   \end{equation*}
   and that for $j^2<\lambda_0$, we actually have $j^2\le\lambda_0-\alpha$. A calculation using that $P_0(I-P_0)=0$ shows that
   \begin{equation*}
      \begin{aligned}
         E_J(\overline H) &= \biggl(\sum_{j^2\le \lambda_0-\alpha} 1_J(h+j^2)\otimes Q_j + 1_{\{\lambda_0\}}(h+n^2)\otimes Q_n\biggr)(I-P_0) \\
         &= \biggl(\sum_{j^2\le \lambda_0-\alpha} 1_J(h+j^2)\otimes Q_j\biggr)(I-P_0).
      \end{aligned}
   \end{equation*}
   Moreover,
   \begin{equation*}
      \begin{aligned}
         E_J(\overline H) (-\partial_z^2 + V) E_J(\overline H) &= \biggl( \sum_{j^2\le \lambda_0-\alpha} (h 1_J(h+j^2))\otimes Q_j \biggr)(I-P_0) \\
         &\ge
         \biggl( \sum_{j^2\le\lambda_0-\alpha} (\lambda_0-j^2-\delta) 1_J(h+j^2))\otimes Q_j \biggr)(I-P_0) \\
         &\ge
         (\alpha-\delta)  \biggl( \sum_{j^2 \le \lambda_0-\alpha}  1_J(h+j^2))\otimes Q_j \biggr)(I-P_0) \\
         &= (\alpha-\delta) E_J(\overline H)
      \end{aligned}
   \end{equation*}

   Note that
   \begin{equation*}
      iB_1 E_{\overline H}(J)=((-\partial_z^2 + V)+i[V,A]-V+i[P_0,A]-P_0)E_{\overline H}(J)=
     ((-\partial_z^2 +V)+K)E_{\overline H}(J),
   \end{equation*}
   where $K$ is compact. Hence
   \begin{equation*}
      E_J(\overline H) iB_1 E_J(\overline H) =  E_J(\overline H) (-\partial_z^2 + V) E_J(\overline H) + E_J(\overline H) K E_J(\overline H) \ge (\alpha-\delta) E_J(\overline H) + E_J(\overline H)KE_J(\overline H),
   \end{equation*}
   and $(H3)_k$ follows from \cite{JMP84} and an argument similar to the one in Example \ref{Ex:line} which converts $(1+A^2)^{s/2}$ to $(1+z^2)^{s/2}$.

   Condition (H4) is verified in the same way as in Example \ref{Ex:line}. Since
   $0<\lambda_0^{1/2}$ and $\lambda_0^{1/2}$ is not an integer, $\Ran \delta(H_0-\lambda)$ is the span of the functions
   $e^{ikz}\cos(j\theta)$, where $j^2+k^2=\lambda_0$, $j=0,\dots, n-1$, $k\in \mathbb R$. The number of such functions is $2n$, so
   $\dim\Ran \delta(H_0-\lambda_0)=:m =2n$. Now we proceed as in Example \ref{Ex:line} and conclude that
   the multiplicity of the continuous spectrum is constant when $\lambda_0\ne j^2$ for $j\in \mathbb Z$.

   To prove (H5), we note that the set $Z_f$ of zeroes of $f$ is at most countable and has no accumulation points, and so the
   zero set $Z_F$ of $F$ is
   \begin{equation*}
       (Z_f\times S^1) \cup (\mathbb R \times \{(1/2 + j)\pi/n;\; j=0,\dots,2n-1\}),
   \end{equation*}
   which is a union of
   straight lines and circles which do not accumulate. Let
   \begin{equation*}
      Z_F^\epsilon:= \{(z,\theta)\in \mathbb R\times S^1;\;|F(z,\theta)|<\epsilon\},
   \end{equation*}
   and let $\chi_\epsilon$ be a smooth cutoff function such that $0\le \chi_\epsilon\le 1$ and
   \begin{equation*}
      \chi_\epsilon(z,\theta) =
      \begin{cases}
         0 &\text{for }(z,\theta)\in Z_F^\epsilon, \\
         1 &\text{for }(z,\theta)\notin Z_F^{2\epsilon}.
      \end{cases}
   \end{equation*}
   We also require that $\chi_\epsilon$ is even in the $\theta$-variable. Let $\varphi\in \mathcal H_s$, and let
   \begin{equation*}
      W_\epsilon(z,\theta):=
      \begin{cases}
         \displaystyle\frac{\chi_\epsilon(z,\theta) \varphi(z,\theta)}{F(z,\theta)} &\text{if }F(z,\theta)\ne 0, \\
         0 &\text{if }F(z,\theta)=0.
      \end{cases}
   \end{equation*}
   Then $W_\epsilon\in X_s$ and $\|W_\epsilon F - \varphi\|_{\mathcal H_s}\to 0$ as $\epsilon\to 0$. It follows that
   $\{W F;W\in X_s\}$ is dense in $\mathcal H_s$, and so (H5) holds.

\end{Ex}

\section{On perturbations of degenerate embedded eigenvalues}\label{S:degenerate}\noindent
In this section we assume that (H1), (H2), and $\text{(H3)}_k$,  are satisfied. In the last theorem of this section we also assume (H4) and (H5'). Let $n$ denote the multiplicity of the
eigenvalue $\lambda_0$. We assume that $n\ge 2$.

We start by fixing an orthonormal basis $\psi_1,\dots,\psi_n$ in the eigenspace of $H$ at $\lambda_0$. We assume that all
$\psi_i$ are real. We first show that this is possible: Suppose that $\psi_1, \dots, \psi_n$ is an ON-basis for $\Ran P_0$, If
$\psi_1$ is not real, then we replace $\psi_1$ by $(\psi_1+C\psi_1)/2$ or $(\psi_1-C\psi_1)/(2i)$, and choose a new ON-basis $\psi_1,\dots, \psi_n$ for $\Ran P_0$, where $C\psi_1=\psi_1$. Now
\begin{equation*}
   C:\span\{\psi_2,\dots,\psi_n\}\to \span\{\psi_2,\dots,\psi_n\},
\end{equation*}
since if $C\psi_2=\alpha \psi_1+\psi_\perp$ where $\langle \psi_\perp,\psi_1\rangle=0$, then
\begin{equation*}
  \langle C\psi_1,C\psi_2\rangle = \langle \psi_2,\psi_1\rangle = 0 = \langle \psi_1,C\psi_2\rangle = \overline{\alpha},
\end{equation*}
and so $\alpha=0$. Now we replace $\psi_2$ by $(\psi_2+C\psi_2)/2$ or $(\psi_2-C\psi_2)/(2i)$ then renormalize the result and repeat. This shows that
$\psi_1,\dots,\psi_n$ can be assumed to be real.

Denote by $P_i$ the orthogonal projection in $\mathcal H$ onto $\span\{\psi_i\}$. Recall from
Section~\ref{S:main} that $P_0$ is the orthogonal projection onto the full eigenspace of $H$ at $\lambda_0$ and
$\overline H = H + P_0$.

Let $H_1:= H + P_0 - P_1 = \overline H - P_1$, and note that $\lambda_0$ is an embedded eigenvalue of multiplicity~$1$
of $H_1$, and that $\psi_1$ is a corresponding normalized eigenvector.

We give sufficient conditions for $H+W$ to have at least one embedded eigenvalue close to $\lambda_0$. The eigenvalue
and eigenvector of $H+W$ that we will construct coincide with the eigenvalue and eigenvector of the operator $H_1+W$.

To this end we first notice that by the proof of Proposition \ref{P:A-equation} there exists a $C^k$ real valued function $\lambda_1(W)$, defined in a neighborhood of $0$ in $X_s$, such that $\langle \psi_1,A(\lambda_1(W),W)\psi_1\rangle = 1$, $\lambda_1(0)=\lambda_0$ and $\lambda'(0)\widetilde W = \langle \psi_1,\widetilde W\psi_1\rangle$, where as usual $A(\lambda,W)$ denotes the operator \eqref{E:A-def} associated with $H$ and the eigenvalue $\lambda_0$.
\begin{Rem}\label{R:A-A_1}
   Note that the operator $Q_1(\lambda+i0,W)$ which is the operator $Q(\lambda+i0,W)$ corresponding to $H_1$ is given by
   $Q_1(\lambda+i0,W) = P_1(H_1 + P_1+ W - \lambda-i0)^{-1}P_1 = P_1 (\overline H + W -\lambda-i0)^{-1} P_1$.
   It follows that $\langle \psi_1,A(\lambda,W)\psi_1\rangle = \langle \psi_1,A_1(\lambda,W)\psi_1\rangle$, where $A_1(\lambda,W)$ denotes the operator \eqref{E:A-def} associated with the operator $H_1$ and the eigenvalue $\lambda_0$.
\end{Rem}

\begin{Prop}\label{P:H_1-eigenvalue}
   Let $n\ge 2$, and assume that (H1), (H2), and $\text{(H3)}_k$  are satisfied for some $k \ge 1$.  Suppose $s > s_1$.  If $W \in X_s$ is sufficiently small, then $H_1+W$ has an embedded eigenvalue in a neighborhood of $\lambda_0$
   if and only if
   \begin{equation*}
      \delta(\overline H + W - \lambda_1(W))\psi_1 = 0,
   \end{equation*}
   where $\lambda_1$ is the function defined above. Moreover the corresponding eigenvector is given by
   \begin{equation*}
      \psi_1^W = (\overline H + W - \lambda_1(W)-i0)^{-1} \psi_1.
   \end{equation*}
\end{Prop}
\begin{proof}
   The result follows directly from equation \eqref{E:deltaA-equation}, Proposition \ref{P:A-equation}, Proposition \ref{P:eigenvector} and Remark \ref{R:A-A_1}.
\end{proof}

\begin{Thm}\label{T:l eig suff}
   Suppose that (H1), (H2), and $\text{(H3)}_k$ are satisfied for some $k \ge 1$. Let $n\ge 2$ and $s>s_1$, and let
   $\lambda_1$ be the function defined above. Then there exists a neighborhood $\mathcal O$ of $0$ in $X_s$ such
   that if $W\in \mathcal O$, then
   a sufficient condition that $\lambda_1(W)$ is an eigenvalue of $H+W$
   is that
   \begin{equation}\label{E:embeigcond1}
      \langle \psi_i,A(\lambda_1(W),W)\psi_1\rangle = 0, \qquad i=2,\dots,n,
   \end{equation}
   and
   \begin{equation}\label{E:embeigcond2}
      \delta(\overline H + W - \lambda_1(W))\psi_1 = 0.
   \end{equation}
\end{Thm}
\begin{proof}
   By \eqref{E:embeigcond2} and Proposition \ref{P:H_1-eigenvalue}, $\lambda_1(W)$ is an eigenvalue of $H_1+W$
   with corresponding eigenfunction $\psi_1^W$.

   The conditions \eqref{E:embeigcond1}  and \eqref{E:embeigcond2} together with Proposition \ref{P:H_1-eigenvalue} imply that
   \begin{equation}\label{E:orthogonal}
      \langle \psi_i,\psi_1^W\rangle = \langle\psi_i,(\overline H + W - \lambda_1(W)-i0)^{-1}\psi_1\rangle = \langle \psi_i,A(\lambda_1(W),W)\psi_1\rangle =  0
   \end{equation}
   for $i=2,\dots,n$ and $W$ sufficiently small. In particular
   the eigenvector $\psi_1^W$ of
   $H_1+W$ is orthogonal to $\psi_i$, $i\ne 1$. Finally we note that $\lambda_1(W)$ is also an eigenvalue of $H+W$, and that
   $\psi_1^W$ is a corresponding eigenvector.
   Indeed,
   \begin{equation*}
      (H + W - \lambda_1(W))\psi_1^{W} = (H_1 + W - \lambda_1(W))\psi_1^{W} -
      \sum_{i=2}^n P_i\psi_1^{W} = 0.
   \end{equation*}
\end{proof}

For our final theorem in this section, we need the additional conditions (H4) and (H5'):
\begin{Thm}\label{T:codim-degenerate}
   Suppose that (H1), (H2), $\text{(H3)}_k$, (H4) and (H5') are satisfied for some $k \ge 1$, and suppose that $n\ge 2$ and $s>s_1$.
   Then there exists a $C^k$ manifold $\mathcal M\subset X_s$  of codimension $\nu:=m+n-1$, a neighborhood $\mathcal O$ of $0\in
   X_s$ and a $\delta>0$
   such that if $W \in \mathcal M \cap \mathcal O$ then $H+W$ has an embedded eigenvalue $\lambda_1(W)\in (\lambda_0-\delta,\lambda_0+\delta)$.
\end{Thm}
\begin{proof}
   The manifold $\mathcal M$ will be the set of $W\in X_s$ such that  \eqref{E:embeigcond1} and \eqref{E:embeigcond2} are
   satisfied. Let $\varphi_i$, $f_i$ and $g_i$, $i=1,\dots,m$ be the functions defined in the proof of Theorem \ref{T:main}:
   The vectors $\varphi_i\in \mathcal H_s$, $i=1,\dots, m$ are chosen so that $f_i=\delta(\overline H-\lambda_0)\varphi_i$ are real and span $\Ran(\delta(\overline H-\lambda_0))$. Then the vectors $g_i$, $i=1,\dots,m$ are chosen to be real and so that they satisfy   $\langle f_j,g_l\rangle=0$ for $j, l=1,\dots,m$. We also have to make sure that the first eigenvector $\psi_1\in \Ran P_0$ is chosen so that  (H5') is satisfied for this $\psi_1$. This means that the other basis vectors which were chosen in the beginning of this section may have to be modified so that $\psi_1,\dots,\psi_n$ form an ON-basis.

   As in the proof of Theorem \ref{T:main}, \eqref{E:embeigcond2} is equivalent to $\langle g_i,\delta(\overline H + W -
   \lambda_1(W))\psi_1\rangle = 0$, $i=1,\dots,m$ and hence
   we need to study the equations
   \begin{equation}\label{E:solvability-cond}
      \begin{aligned}
         \langle \psi_i,A(\lambda_1(W),W)\psi_1\rangle &= 0, \qquad i=2,\dots,n, \\
         \langle g_i,\delta(\overline H+W-\lambda_1(W))\psi_1\rangle &=0, \qquad  i=1,\dots,m.
      \end{aligned}
   \end{equation}
   Note that there are $\nu$ conditions to be satisfied. We write \eqref{E:solvability-cond} as
   \begin{equation*}
      F(W) = 0,
   \end{equation*}
   where $F$ maps a neighborhood of $0$ in $X_s$ to $\mathbb R^\nu$, and the components of $F$ are the left hand side of the equations
   \eqref{E:solvability-cond} in some order.

   By Proposition \ref{P:differentiable}, $F$ is a $C^k$ function of $W$, and
   a calculation shows that the components of $F'(0)$ are given by the functionals
   \begin{equation}\label{E:functionals}
      \begin{aligned}
         W \mapsto -\langle g_i,\delta(\overline H - \lambda_0)W\psi_1\rangle, \qquad & i=1,\dots, m,\\
         W \mapsto -\langle \psi_i,W\psi_1\rangle, \qquad &  i=2,\dots,n.
      \end{aligned}
   \end{equation}
   We must show that these functionals are linearly independent, and so we let
   $g:=\sum_{i=1}^m\alpha_i g_i$ and $\psi_\perp:= \sum_{i=2}^n \beta_i \psi_i$. Then
   \begin{equation*}
      \langle g,\delta(\overline H - \lambda_0)W\psi_1\rangle + \langle \psi_\perp,W\psi_1\rangle = 0
   \end{equation*}
   for every $W\in X_s$  holds  if and only if
   \begin{equation*}
      \langle \delta(\overline H - \lambda_0) g+\psi_\perp,W\psi_1\rangle = 0
   \end{equation*}
   for every $W\in X_s$. If this is true, then by (H5'),
   \begin{equation*}
      \psi_\perp+\delta(\overline H-\lambda_0)g=0.
   \end{equation*}
   But
   \begin{equation*}
      \langle \psi_\perp,\delta(\overline H-\lambda_0)g\rangle=\langle \delta(\overline H-\lambda_0)\psi_\perp,g\rangle =0
   \end{equation*}
   since $H\psi_\perp = \lambda_0 \psi_\perp$. Thus $\psi_\perp=0$ (and hence $\beta_i=0$)
   and $\delta(\overline H - \lambda_0) g =0$. But
   \begin{equation*}
      \langle g_j,\delta(\overline H-\lambda_0)\varphi_i\rangle=\delta_{ij}.
   \end{equation*}
   Thus
   \begin{equation*}
      \langle g, \delta(\overline H-\lambda_0)\varphi_i\rangle = \alpha_i = \langle\delta(\overline H-\lambda_0)g,\varphi_i\rangle=0,
   \end{equation*}
   which shows that the functionals in \eqref{E:functionals} are linearly independent.

   Finally, we make the decomposition $X_s = (\ker F'(0))\oplus \mathcal V$, where $\mathcal V$ has dimension $\nu$ and the
   map $F'(0):\mathcal V \to \mathbb R^\nu$ is a linear homeomorphism. For $W\in X_s$, we write $W = \xi +
   \eta$ where $\xi\in \ker F'(0)$ and $\nu \in \mathcal V$. By the implicit function theorem, there is a neighborhood $U$ of $0 \in \ker F'(0)$ and a $C^k$ function
   $\eta:U \rightarrow \mathcal V$,
   $\xi\mapsto\eta(\xi)$ such that $\eta(0)=0$ and $F(\xi +\eta(\xi))= 0$ for $\|\xi\|_{X_s}$ small.
\end{proof}
   \begin{Ex}
   We revisit Example \ref{Ex:cylinder} when $H:= -\Delta + V$ on $L^2(\mathbb R\times S^1)$, but this time we do not restrict to the subspace of functions which are even in the $\theta$ variable. The multiplicity of the eigenvalue $\lambda_0$ is $2$. Let $s=k+1/2$, where $k$ is a positive integer. The conditions (H1)--(H4) are verified as in Example \ref{Ex:cylinder}, except that the multiplicity of the continuous spectrum is
   now $m:= 4n-2$ since $(n-1)^2<\lambda_0<n^2$. Let $f$ be as in Example \ref{Ex:cylinder}. Following the notation of this section, we choose
   \begin{equation*}
       \begin{aligned}
          \psi_1(z,\theta) &= \frac{1}{\sqrt{\pi}\|f\|_{L^2(\mathbb R)}}f(z) \cos(n\theta), \\
          \psi_2(z,\theta) &= \frac{1}{\sqrt{\pi}\|f\|_{L^2(\mathbb R)}}f(z) \sin(n\theta).
       \end{aligned}
   \end{equation*}
   Then $\psi_1$ and $\psi_2$ are normalized eigenfunctions with eigenvalue $\lambda_0=n^2+e>0$.

   Theorem \ref{T:codim-degenerate} guarantees the existence of a $C^k$ manifold $\mathcal M$ of codimension $\nu:=m+1=4n-1$ such that if $W$ belongs to this manifold and is sufficiently small, then $H+W$ has an embedded eigenvalue close to
   $\lambda_0$.

   Note that with a different choice of $\psi_1$, we would in general get a different manifold $\mathcal M$, and that there is a $1$ parameter family of such normalized $\psi_1$. The set of small perturbations making the embedded
   eigenvalue persist is included in the union of these manifolds. The structure of the set of small perturbations making the embedded eigenvalue persist is not yet fully understood in this case.
\end{Ex}
\section{Extensions to the case of infinite multiplicity of the continuous spectrum}\label{S:infinmult}
In this section we study the example of the self-adjoint Schr\"odinger operator $H := -\Delta + V$ in the space $\mathcal H:= L^2(\mathbb R^n)$, $n\ge 2$, in which
case the continuous spectrum may have infinite multiplicity. We impose some conditions on $V$. First we assume that $V$ is a real measurable locally bounded function on $\mathbb R^n$. We denote by $\dot H$ the symmetric operator in $L^2(\mathbb R^n)$ with
$\Dom (\dot H) = C_0^\infty(\mathbb R^n)$ such that $\dot H u = -\Delta u + V u$ for $u\in C_0^\infty(\mathbb R^n)$.

We assume that $\dot H$ is essentially self-adjoint. We note that a simple sufficient condition for $\dot H$ to be
essentially self-adjoint is that $V_-(x):=\min\{V(x),0\}$ is a bounded function. (For general conditions ensuring essential self-adjointness, see \cite{tK72}.)

We denote by $H$ the self-adjoint operator which is the closure of $\dot H$ in $L^2(\mathbb R^n)$. We observe that if $\varphi_0$ is an eigenfunction of $H$, associated with the eigenvalue $\lambda_0$, then $\varphi_0$ is a continuous function. Indeed, since
\begin{equation*}
   \langle (-\Delta +V-\lambda_0)\psi,\varphi_0\rangle = \langle (H-\lambda_0)\psi,\varphi_0\rangle = \langle\psi,(H-\lambda_0)\varphi_0\rangle = 0
\end{equation*}
for all $\psi\in C_0^\infty (\mathbb R^n)\subset \Dom(H)$ and since $V\in L_{loc}^\infty(\mathbb R^n)$, it follows, using the $L^p$ regularity theory of weak solutions of elliptic equations, that the eigenfunction $\varphi_0$ belongs to the Sobolev space $W_{loc}^{2,p}(\mathbb R^n)$ for any $p$, $1<p<\infty$ (see e.g. Theorem 6.1. in \cite{sA59}). It then follows from the Sobolev embedding theorem that the eigenfunction $\varphi_0$ is continuous. (More precisely, it follows that $\varphi_0\in C^1(\mathbb R^n)$
and that the first order derivatives of $\varphi_0$ satisfy a local H\"older condition of any order $<1$.)

We consider the Schr\"odinger operator $H$ in the setup of section \ref{S:main}. We choose for the $\mathcal H_s$ spaces the weighted $L^2$ spaces on $\mathbb R^n$ with weight $(1+|x|^2)^s$ and we let the antiunitary involution  $C$ be complex conjugation. We shall show in the following that under assumptions (H2)--$\text{(H3)}_k$, embedded eigenvalues of $H$ persists for a large class of perturbations $W$.

In the following, we denote by $\lambda_0$ some fixed embedded eigenvalue of $H$ verifying assumption (H2). As usual, $P_0$ denotes the orthogonal projection on the eigenspace at $\lambda_0$. We denote by $\varphi_0(x)$ some fixed eigenfunction corresponding to $\lambda_0$. We assume that $\varphi_0$ is real and normalized. We also fix some ball $B_r(x_0):=\{x\in\mathbb R^n;\; |x-x_0|<r\}$
such that $\varphi_0(x)\ne 0$ on $\overline {B_r(x_0)}$.
We introduce the following function spaces:
\begin{equation*}
   \begin{aligned}
      K_0 &:= \{ f\in L^{\infty}(\mathbb R^n);\; f \, \text {is real and} \, f=0 \text{ a.e. in the complement of }B_r(x_0)\}, \\
      K_1&:=\{f\in C^2(\mathbb R^n);\; f \, \text {is real} \,, \,\supp f \subset \overline{B_r(x_0)}\text{ and }\langle f,\varphi\rangle = 0 \text{ for every }\varphi\in \Ran P_0\}.
   \end{aligned}
\end{equation*}
$K_0$ and $K_1$ are considered as real Banach spaces with norms
\begin{equation*}
   \|f\|_{K_0} := \|f\|_{L^\infty(\mathbb R^n)}
\end{equation*}
and
\begin{equation*}
   \|f\|_{K_1} := \sup_{x\in \mathbb R^n} \sum_{|\alpha|\le 2} \left|\left(\frac{\partial}{\partial x}\right)^\alpha f(x)\right|.
\end{equation*}
\begin{Thm}
   Let $H:= -\Delta + V$ and let $\lambda_0$ and $\varphi_0$ be as above. Assume that (H2)--$(H3)_k$ are satisfied.
   Then there exist positive numbers $\delta$, $\eta$, a $C^k$ injective map
   \begin{equation*}
      g:\{u\in K_1;\; \|u\|_{K_1}<\delta\}\to K_{0,\eta}:=\{W\in K_0;\; \|W\|_{K_0}<\eta\},
   \end{equation*}
   and a $C^k$ map
   \begin{equation}\label{E:lambda-def}
      \lambda:K_{0,\eta}\to \mathbb R
   \end{equation}
    such that
   $g(0)=0$, $\lambda(0)=\lambda_0$ and if $W=g(u)$ and $\|u\|_{K_1}<\delta$ then $\lambda(W)$ is an eigenvalue of $H+W$.
\end{Thm}
\begin{proof}
   We note that the analysis and results of Sections \ref{S:prel}, \ref{S:Q} and \ref{S:degenerate} (except for Theorem \ref{T:codim-degenerate}) are valid also when the continuous spectrum of $H$ has infinite multiplicity in a
   neighborhood of $\lambda_0$. This leads us to define the function $\lambda=\lambda(W)$ (the map \eqref{E:lambda-def}) to be the solution $\lambda=\lambda(W)$ of the equation
   \begin{equation}\label{E:lambda-eq}
      \langle \varphi_0,A(\lambda,W)\varphi_0\rangle = 1, \qquad \lambda(0)=\lambda_0
   \end{equation}
   for $W\in K_{0,\eta}$ for some $\eta>0$.
   \begin{Rem} Here and in the following, $\eta$ denotes a generic small positive number which may change throughout the proof. All statements involving $\eta$ hold under the assumption that $\eta$ is chosen sufficiently small.
   \end{Rem}
   Now, if $\lambda$ is a simple eigenvalue, then it follows from Proposition \ref{P:A-equation} that there exists a unique solution $\lambda=\lambda(W)$ of \eqref{E:lambda-eq} for all $W\in K_{0,\eta}$, where $\eta$ is sufficiently small, such that $\lambda(W)$ is of class $C^k$ in $K_{0,\eta}$. Moreover, by the proof of Proposition \ref{P:A-equation}, the same result holds if $\lambda_0$ is a degenerate eigenvalue.

   Next, let $F:K_1\times K_{0,\eta} \to K_0$ be defined by
   \begin{equation*}
      F(u,W) := W \varphi_0 - (H + W - \lambda(W)) u.
   \end{equation*}
   Note that $F$ is well defined since $C_0^2(\mathbb R^n)\subset \Dom(H)$ and $\varphi_0$ is continuous. Using that $\lambda\in C^k$, it follows that
   $F\in C^k(K_1\times K_{0,\eta})$. A short calculation shows that
   \begin{equation*}
      F'_W (0,0) \widetilde W = \widetilde W \varphi_0.
   \end{equation*}
   From the definition of $K_0$, it follows that the map $F'_W(0,0):K_0\to K_0$ is invertible. Hence by the implicit function theorem,
   there exist $\delta>0$ and a $C^k$ map $g:\{u\in K_1;\; \|u\|_{K_1}<\delta\}\to K_{0,\eta}$ such that $F(u,W)=0$ for $\|u\|_{K_1}<\delta$ and $W\in K_{0,\eta}$ if and only if $W=g(u)$.

   We claim that $\lambda(W)$ with $W=g(u)$, $\|u\|_{K_1}<\delta$ is an eigenvalue of $H+W$. For the claim to hold, we need to show
   in the case that $\lambda_0$ is a simple eigenvalue, that the two equations \eqref{E:deltaA-equation} with $f=\varphi_0$ and $W=g(u)$ hold. Now the second equation \eqref{E:deltaA-equation} in our case coincides with equation of \eqref{E:lambda-eq}. Thus, to prove the result for a simple eigenvalue $\lambda_0$, we only need to verify that the first equation of \eqref{E:deltaA-equation} holds; i.e. we need to show that
   \begin{equation}\label{E:delta-eq}
      \delta(\overline H+W-\lambda(W))\varphi_0 = 0
   \end{equation}
   for $W=g(u)$, $\|u\|_{K_1}<\delta$.

   To prove the claim when $\lambda_0$ is a degenerate eigenvalue, we observe that it follows from Proposition \ref{P:H_1-eigenvalue}
   and Theorem \ref{T:l eig suff} (with $\psi_1$ replaced by $\varphi_0$  and $\lambda_1(W)$ replaced by $\lambda(W)$) that $\lambda(W)$ with $W=g(u)$, $\|u\|_{K_1}<\delta$, is an eigenvalue of $H+W$ if \eqref{E:delta-eq} holds and in addition:
   \begin{equation}\label{E:A_1-eq}
      \langle \varphi,A(\lambda(W),W)\varphi_0\rangle = 0
   \end{equation}
   for all $\varphi\in \Ran P_0$ such that $\langle \varphi,\varphi_0\rangle =0$.

   We proceed to show that \eqref{E:delta-eq} and \eqref{E:A_1-eq} hold, thus proving our claim. To this end, we consider the functions
   $(\overline H + W - \lambda(W)\pm i0)^{-1}\varphi_0$ for $W=g(u)$, $\|u\|_{K_1}<\delta$ ($\delta$ small as above). Using the second resolvent equation, we find that
   \begin{equation}\label{E:resolvent-eq-mult}
      \begin{aligned}
         (\overline H &+ W - \lambda(W)\pm i0)^{-1}\varphi_0 \\
         &= (\overline H -\lambda(W)\pm i0)^{-1}\varphi_0 - (\overline H + W - \lambda(W)\pm i0)^{-1}W (\overline H - \lambda(W)\pm i0)^{-1}\varphi_0 \\
         &=(1+\lambda_0-\lambda(W))^{-1}\varphi_0 -(1+\lambda_0 -\lambda(W))^{-1}(\overline H + W - \lambda(W)\pm i0)^{-1}W\varphi_0 \\
         &= (1+\lambda_0-\lambda(W))^{-1}\varphi_0 - (1+\lambda_0-\lambda(W))^{-1}(\overline H + W - \lambda(W)\pm i0)^{-1}(\overline H + W - \lambda(W))u \\
         &= (1+\lambda_0-\lambda(W))^{-1}(\varphi_0 - u),
      \end{aligned}
   \end{equation}
   where the third equality of \eqref{E:resolvent-eq-mult} follows since $W\varphi_0 = (\overline H + W -\lambda(W))u$ for $W=g(u)$, which
   follows from the definition of $g$ since $P_0 u =0$. The last equality in \eqref{E:resolvent-eq-mult} follows since $u$ has compact support, and thus $u\in\mathcal H_s$ for all $s$.

   Now it follows from \eqref{E:resolvent-eq-mult} that $(\overline H + W -\lambda(W)+i0)^{-1}\varphi_0 = (\overline H + W - \lambda(W) - i0)^{-1} \varphi_0$ for $W=g(u)$, where $\|u\|_{K_1}<\delta$ for $\delta$ sufficiently small, which implies (by the definition of $\delta(\cdot)$) that \eqref{E:delta-eq} holds. Also, let $\varphi\in \Ran P_0$ and assume that $\langle \varphi,\varphi_0\rangle=0$. Using the definition of $A(\cdot,\cdot)$ and \eqref{E:resolvent-eq-mult} we find that
   \begin{equation*}
      \begin{aligned}
         \langle \varphi,A(\lambda(W),W)\varphi_0\rangle &= \frac{1}{2}\langle \varphi,\left((\overline H + W - \lambda(W)-i0)^{-1} +
         (\overline H + W - \lambda(W)+ i0)^{-1})^{-1}\right)\varphi_0\rangle \\
         &=(1+\lambda_0-\lambda(W))^{-1}\langle\varphi,\varphi_0-u\rangle = 0,
      \end{aligned}
   \end{equation*}
   since $\langle \varphi,\varphi_0\rangle = 0$ by assumption and $\langle \varphi,u\rangle =0$ by the definition of $K_1$. This completes the proof that $\lambda(W)$ is an eigenvalue of $H+W$ for $W=g(u)$, $\|u\|_{K_1}<\delta$, $\delta$ sufficiently small.

   It remains to check that the map $g:\{u\in K_1;\; \|u\|_{K_1}<\delta\}\to K_0$ is injective. If $g(u_1) = g(u_2) = W$ then
   \begin{equation*}
      0 = F(u_2,W)-F(u_1,W) = (H+W-\lambda(W))(u_1-u_2),
   \end{equation*}
   i.e. $u_1 - u_2$ is an eigenfunction of $H + W$ with eigenvalue $\lambda(W)$. But since $P_0 u_1 = P_0 u_2=0$, it follows that
   $(H+W-\lambda(W))(u_1-u_2) = (\overline H+W-\lambda(W))(u_1-u_2)$, and so $u_1 - u_2$ is also an eigenfunction of $\overline H + W$.
   Since $\overline H + W$ has no eigenvalues in a neighborhood of $\lambda_0$, the only possibility is that
   $u_1-u_2 = 0$. Hence $g$ is injective.
\end{proof}

\section{Acknowledgment}

I.H. would like to thank Tom Kriete for a useful conversation.

\end{document}